\documentclass{amsart}
\usepackage[ascii]{inputenc}

\usepackage[english]{babel}
\usepackage{booktabs}
\usepackage[sort]{cite}
\usepackage{amssymb}
\usepackage{url}
\usepackage{graphics}

\newcommand{\Var}[1]{\mathcal{#1}}
\newcommand{\Sec}[2]{\sigma_{#1}(#2)}
\newcommand{\SecZ}[2]{\sigma^0_{#1}(#2)}
\newcommand{\tensor}[1]{\mathfrak{#1}}
\newcommand{\vect}[1]{\mathbf{#1}}
\newcommand{\sten}[3]{\vect{#1}_{#2}^{#3}}
\newcommand{\Tang}[2]{\mathrm{T}_{#1} {#2}}

\newcommand{\party}{\sqcup}

\newcommand{\Pj}{\mathbb{P}}
\newcommand{\C}{\mathbb{C}}
\newcommand{\R}{\mathbb{R}}
\newcommand{\F}{\mathbb{F}}
\newcommand{\oo}{\mathcal{O}}
\newcommand{\ii}{\mathcal{I}}

\newcommand{\refthm}[1]{{Theorem \ref{#1}}}
\newcommand{\reflem}[1]{{Lemma \ref{#1}}}
\newcommand{\refeqn}[1]{{(\ref{#1})}}

\newcommand{\refsec}[1]{{Section \ref{#1}}}

\newcommand{\refcor}[1]{{Corollary \ref{#1}}}

\newcommand{\refprop}[1]{{Proposition \ref{#1}}}

\newcommand{\refconj}[1]{{Conjecture \ref{#1}}}
\newcommand{\refrem}[1]{{Remark \ref{#1}}}

\newcommand{\rank}[1]{\operatorname{rank}(#1)}

\newtheorem{theorem}{Theorem}
\newtheorem{lemma}[theorem]{Lemma}
\newtheorem{proposition}[theorem]{Proposition}
\newtheorem{corollary}[theorem]{Corollary}
\newtheorem{conjecture}[theorem]{Conjecture}
\newtheorem{remark}[theorem]{Remark}

\newtheorem{romanthm}{Lemma}

\newtheorem{romancor}[romanthm]{Corollary}

\theoremstyle{definition}
\newtheorem{definition}[theorem]{Definition}
\newtheorem{example}[theorem]{Example}

\newtheorem{romanrem}[romanthm]{Remark}
\newtheorem{romanex}[romanthm]{Example}

\title[Effective criteria for identifiability]{Effective criteria for specific identifiability of tensors and forms}

\author{Luca Chiantini}
\address{Luca Chiantini: Dipartimento di Ingegneria dell'Informazione e Scienze Matematiche, 
Universit\`a di Siena, Italy.}
\email{luca.chiantini@unisi.it}

\author{Giorgio Ottaviani}
\address{Giorgio Ottaviani: Dipartimento di Matematica e Informatica ``Ulisse Dini'', Universit\`a di Firenze, Italy.}
\email{ottavian@math.unifi.it}
\thanks{The first and second author are members of the Italian GNSAGA-INDAM}

\author{Nick Vannieuwenhoven}
\address{Nick Vannieuwenhoven: KU Leuven, Department of Computer Science, Belgium}
\email{nick.vannieuwenhoven@cs.kuleuven.be}
\thanks{The third author was supported by a Postdoctoral fellowship of the 
Research Foundation--Flanders (FWO)}

\allowdisplaybreaks

\keywords{tensor rank decomposition, Waring decomposition, effective identifiability, reshaped Kruskal criterion, Hilbert function}
\subjclass[2000]{15A69, 15A72, 14N20, 14N05, 14Q15, 14Q20}

\begin{document}

\maketitle

\begin{abstract}
In applications where the tensor rank decomposition arises, one often relies on its identifiability properties
for interpreting the individual rank-$1$ terms appearing in the decomposition. Several criteria for 
identifiability have been proposed in the literature, however few results exist on how frequently 
they are satisfied. We propose to call a criterion effective if it is satisfied on a dense, open subset 
of the smallest semi-algebraic set enclosing the set of rank-$r$ tensors. 
We analyze the effectiveness of Kruskal's criterion when it is combined with reshaping. It is proved that this criterion is effective for both real and 
complex tensors in its entire range of applicability, which is usually much smaller than
the smallest typical rank. Our proof explains when reshaping-based algorithms for computing tensor rank decompositions may be expected to recover the decomposition.
Specializing the analysis to symmetric tensors or forms reveals that the reshaped Kruskal criterion may even be effective up to the smallest typical rank for some third, fourth and sixth order symmetric tensors of small dimension as well as for binary forms of degree at least three. 
We extended this result to $4 \times 4 \times 4 \times 4$ symmetric tensors by analyzing the Hilbert function, resulting in a criterion for symmetric identifiability that is effective up to symmetric rank $8$, which is optimal.
\end{abstract}

\section{Introduction}
A tensor rank decomposition expresses a tensor $\tensor{A} \in \F^{n_1} \otimes \F^{n_2} \otimes \cdots \otimes \F^{n_d}$ as a linear combination of rank-1 tensors, as follows:
\begin{align} \label{eqn_rank_decomposition}
 \tensor{A} = \sum_{i=1}^r \sten{a}{i}{1} \otimes \sten{a}{i}{2} \otimes \cdots \otimes \sten{a}{i}{d},
\end{align}
where $\sten{a}{i}{k} \in \F^{n_k}$, and $\F$ is either the real field $\R$ or complex field $\C$. When $r$ is minimal in the above expression, then it is called the \emph{rank} of $\tensor{A}$. 
A key property of the tensor rank decomposition is its \emph{generic identifiability} \cite{BCO2013,COV2014,DDL2015}. This means that the expression \refeqn{eqn_rank_decomposition} is unique up to a permutation of the summands and scaling of the vectors on a dense open subset of the set of tensors admitting an expression as in \refeqn{eqn_rank_decomposition}. This uniqueness property renders it useful in several applications. For instance, in chemometrics, decomposition \refeqn{eqn_rank_decomposition} arises in the simultaneous spectral analysis of unknown mixtures of fluorophores, where the tensor rank decomposition of the corresponding tensor reveals the emission-excitation matrices of the individual chemical molecules in the mixtures, hence allowing a trained chemist to identify the fluorophores \cite{Appellof1981}. 

Another application of tensor decompositions is parameter identification in statistical models with hidden variables, such as principal component analysis (or blind source separation), exchangeable single topic models and hidden Markov models. Such applications were recently surveyed in a tensor-based framework in \cite{Anandkumar2014b}. The key in these applications consists of recovering the unknown parameters by computing a \emph{Waring decomposition} of a higher-order moment tensor constructed from the known samples. In other words, one seeks a decomposition
\begin{align}\label{eqn_waring_decomposition}
 \tensor{A} = \sum_{i=1}^r \lambda_i \sten{a}{i}{} \otimes \cdots \otimes \sten{a}{i}{} = \sum_{i=1}^r \lambda_i \sten{a}{i}{\otimes d},
\end{align}
where $\sten{a}{i}{} \in \F^{n}$ and $\lambda_i \in \F$ for all $i=1,\ldots,r$. Note that $\tensor{A}$ is a \emph{symmetric} tensor in this case. If $r$ is minimal, then $r$ is called the Waring or symmetric rank of $\tensor{A}$. Uniqueness of Waring decompositions is again the key for ensuring that the recovered parameters of the model are unique and interpretable. Generic identifiability of complex Waring decompositions of subgeneric rank for nearly all tensor spaces was proved in \cite{COV2016}.

We address in this paper the problem of \emph{specific identifiability}: given a tensor rank decomposition of length $r$ in the space $\F^{n_1}\otimes\F^{n_2}\otimes\cdots\otimes\F^{n_d}$, prove that it is unique. Let $\Var{S}$ denote the variety of rank-$1$ tensors in this space. As it is conjectured that the generic\footnote{We call $p \in S$ ``generic'' with respect to some property in the set $S$, if the property fails to hold at most for the elements in a strict subvariety of $S$.} tensor of \emph{subtypical}\footnote{Recall that the smallest typical rank over $\R$ coincides with the generic rank over $\C$ \cite{BT2015}, hence the term \emph{subtypical} through all this paper coincides with the term \emph{subgeneric} used in \cite{COV2016}.} rank $r$, i.e.,
\begin{align} \label{eqn_subgeneric}
  r < \overline{r}_\Var{S} = \frac{n_1 n_2 \cdots n_d}{n_1 + n_2 + \cdots + n_d - d + 1} 
\end{align}
has a unique decomposition, provided it is not one of the exceptional cases listed in \cite[Theorem 1.1]{COV2014},
we believe that any practical criterion for specific identifiability must be more informative than the following naive Monte Carlo algorithm:
\begin{enumerate}
 \item[S1.] If the number of terms in the given tensor decomposition is less than $\overline{r}_\Var{S}$, then claim ``Identifiable,'' otherwise claim ``Not identifiable.''
\end{enumerate}
This simple algorithm has a $100\%$ probability of returning a correct result if one samples decompositions of length $r$ from any probability distribution whose support is not contained in the Zariski-closed locus where $r$-identifiability fails (assuming generic $r$-identifiability; see \refsec{sec_identifiability}). It also has a $0\%$ chance of returning an incorrect answer---it \emph{can} be wrong, e.g., if the unidentifiable tensor $\vect{a}\otimes\vect{a}\otimes\vect{a} + \vect{b}\otimes\vect{b}\otimes\vect{a}$ is presented as input, but the probability of sampling these tensors is zero. Deterministic algorithms for specific $r$-identifiability, e.g., \cite{Kruskal1977,DDL2013Part2,COV2014,SB2000,JS2004,Stegeman2009}, merit consideration, however only if they are what we propose to call \emph{effective}: if they can prove identifiability on a dense, open subset of the set of tensors admitting decomposition \refeqn{eqn_rank_decomposition}. A deterministic criterion is thus effective if its conditions are satisfied generically; that is, if the same criterion also proves generic identifiability. Kruskal's well-known criterion for $r$-identifiability is deterministic: it is a sufficient condition for uniqueness. If the criterion is not satisfied, the outcome of the test is inconclusive. Effective criteria are allowed to have such inconclusive outcomes provided that they do not form a Euclidean-open set. It will not surprise the experts that Kruskal's criterion \cite{Kruskal1977} is effective. Domanov and De Lathauwer \cite{DDL2015} recently proved that some of their criteria for \emph{third-order} tensors from \cite{DDL2013Part2} are effective. Presently, only a few effective criteria for specific $r$-identifiability of tensors of higher order, i.e., $d\ge4$, are---informally---known, notably the generalization of Kruskal's criterion to higher-order tensors due to Sidiropoulos and Bro \cite{SB2000}.

In private communication, I.~Domanov remarked that ``in practice, when one wants to check that the [tensor rank decomposition] of a tensor of order higher than $3$ is unique, [one] just reshapes the tensor into a third-order tensor and then applies the classical Kruskal result [...]. The reduction to the third-order case is quite standard and well-known;'' 
indeed this idea appears also in \cite{SB2000,CGG2002,Landsberg2012,PTC2013,SdL2015}.
Formally, it can be stated as follows.
Let $\vect{h}\cup\vect{k}\cup\vect{l}=\{1,2,\ldots,d\}$ be a partition where $\vect{h}$, $\vect{k}$ and $\vect{l}$ have cardinalities $d_1$, $d_2$ and $d_3$ respectively. Let $\Var{S} = \operatorname{Seg}( \F^{n_1} \times \cdots \times \F^{n_d} )$ be the variety of rank-$1$ tensors in $\F^{n_1} \otimes \cdots \otimes \F^{n_d}$, and let $\Var{S}_{\vect{h},\vect{k},\vect{l}} = \operatorname{Seg}( \F^{n_{h_1} \cdots n_{h_{d_1}}} \times \F^{n_{k_1}\cdots n_{k_{d_2}}} \times \F^{n_{l_1}\cdots n_{l_{d_3}}} )$ be the variety of rank-$1$ tensors in the reshaped tensor space. We may consider the natural inclusion
\(
\Var{S} \hookrightarrow \Var{S}_{\vect{h},\vect{k},\vect{l}}
\)
and then apply a criterion for specific $r$-identifiability with respect to $\Var{S}_{\vect{h},\vect{k},\vect{l}}$. If this criterion certifies $r$-identifiability, then it entails $r$-identifiability with respect to $\Var{S}$ as well. 
While this idea is valid, \emph{applying an effective criterion for third-order tensors to reshaped higher-order tensors does not suffice for concluding that it is also an effective criterion for higher-order tensors}. 
Indeed, since $\Var{S}$ has dimension strictly less than $\Var{S}_{\vect{h},\vect{k},\vect{l}}$ one expects that the set of rank-$r$ tensors in $\F^{n_1}\otimes\cdots\otimes\F^{n_d}$ constitutes a Zariski-closed subset of the rank-$r$ tensors in the reshaped tensor space. As a result, the effective criterion for $\Var{S}_{\vect{h},\vect{k},\vect{l}}$ might thus never apply to the elements of $\Var{S} \hookrightarrow \Var{S}_{\vect{h},\vect{k},\vect{l}}$. 
This observation was the impetus for the present work and the reason why our results will always be presented in the general setting.

Our first main result, proved in \refsec{sec_reshaped_kruskal}, can be stated informally as follows.
\begin{theorem}
 Kruskal's criterion applied to a reshaped rank-$r$ tensor is an effective criterion for specific $r$-identifiability.
\end{theorem}

The reshaped Kruskal criterion as well as the criteria in \cite{DDL2013Part2,DDL2015,JS2004} applied 
to a reshaped tensor can all be considered as state-of-the-art results in specific identifiability. 
Nevertheless, combining reshaping with a criterion for lower-order identifiability may not be expected to prove specific identifiability up to the (nearly) optimal value $\overline{r}_\Var{S}-1$. Indeed, consider any partition $\vect{h}_{1} \cup \cdots \cup \vect{h}_{t} = \{1,2,\ldots,d\}$ with $t < d$. Then, 
\[
 \overline{r}_{\Var{S}_{\vect{h}_1,\ldots,\vect{h}_t}} =  \frac{ n_1 n_2 \cdots n_d }{ 1 + \sum_{k=1}^t \bigl( -1 + \prod_{\ell \in \vect{h}_k} n_{\ell} \bigr) } \le \frac{n_1 n_2 \cdots n_d}{n_1 + n_2 + \cdots + n_d - d + 1} = \overline{r}_\Var{S},
\]
where typically the integers $n_i \ge 2$ are such that a strict inequality occurs. 

In the second half of the paper, the reshaped Kruskal criterion is considered for symmetric tensors. Remarkably, in $4$ cases of small dimension as well as for binary forms this criterion is effective in the entire range where generic $r$-identifiability holds. 
We show that an analysis of the Hilbert function yields another case, namely the case of symmetric tensors in $\F^4 \otimes \F^4 \otimes \F^4 \otimes \F^4$. 
The second main result, which is proved in \refsec{sec_first_effective}, can be stated as follows.
\begin{theorem}\label{thm_effective_crit}
Let $S^d \F^n$ be the linear subspace of symmetric tensors in $\F^n \otimes \cdots \otimes \F^n$.
Then, there exist effective criteria for specific symmetric identifiability of symmetric rank-$r$ tensors in $S^3 \F^{n}$ with $n=3,4,5$ and $6$, $S^4 \F^3$, $S^4 \F^4$, $S^6 \F^3$, and $S^d \F^{2}$ for $d \ge 3$ that are effective for every $r \in \mathbb{N}$ where generic $r$-identifiability holds.
\end{theorem}
These are the only cases we know of where effective criteria for specific identifiability exist that can be applied up to the bound for generic identifiability.

The outline of the remainder of this paper is as follows. In the next section, some preliminary material is recalled. The known results about generic identifiability are presented in \refsec{sec_identifiability}. We analyze the reshaped Kruskal criterion in \refsec{sec_reshaped_kruskal}: we prove that it is an effective criterion and present a heuristic for choosing a good reshaping. \refsec{sec_symmetric} presents the variant of the reshaped Kruskal criterion for symmetric tensors, proving most cases of the second main result. It also explains how analyzing the Hilbert function may lead to results about specific identifiability for symmetric tensors. These insights culminate in \refsec{sec_first_effective}, where we complete the proof the second main result, then provide an algorithm implementing that effective criterion, and finally present some concrete examples. In \refsec{sec_applications}, we explain when reshaping-based algorithms for computing tensor rank decompositions will recover the decomposition. \refsec{sec_conclusions} presents our main conclusions.

\subsection*{Notation}
Varieties are typeset in a calligraphic font, tensors in a fraktur font, matrices in upper case, and vectors in boldface lower case. The field $\F$ denotes either $\R$ or $\C$. Projectivization is denoted by $\Pj$. 
$V$ denotes a finite-dimensional vector space over the field $\F$. 
The matrix transpose and conjugate transpose are denoted by $\cdot^T$ and $\cdot^H$ respectively. 
The Khatri--Rao product of $A \in \F^{m \times r}$ and $B \in \F^{n \times r}$ is
\(
 A \odot B = \begin{bmatrix} \vect{a}_1 \otimes \vect{b}_1 & \vect{a}_2 \otimes \vect{b}_2 & \cdots & \vect{a}_r \otimes \vect{b}_r \end{bmatrix}.
\)
A set partition is denoted by $S_1 \party \cdots \party S_k = \{1,\ldots,m\}$. 
If $\Var{X}$ is a variety, then $\Var{X}_0$ is defined as $\Var{X}$ minus the zero element. The affine cone over a projective variety $\Var{X} \subset \Pj\F^n$ is $\widehat{\Var{X}} := \{ \alpha x \;|\; x \in \Var{X}, \alpha \in \F \}$. 
The Segre variety $\operatorname{Seg}(\Pj\F^{n_1} \times \cdots \times \Pj\F^{n_d}) \subset \Pj(\F^{n_1} \otimes \cdots \otimes \F^{n_d})$ is denoted by $\Var{S}$, and the Veronese variety $v_d(\Pj\F^n) \subset \Pj S^d \F^n$ is denoted by $\Var{V}$. The projective dimension of the Segre variety $\Var{S}$ is denoted by $\Sigma = \sum_{k=1}^d (n_k - 1)$. The dimension of $\F^{n_1} \otimes \cdots \otimes \F^{n_d}$ is $\Pi = \prod_{k=1}^d n_k$, and the dimension of $S^d \F^n$ is $\Gamma = \binom{n-1+d}{d}$.

\section{Preliminaries}\label{sec_preliminaries}
We recall some terminology from algebraic geometry; the reader is referred to Landsberg \cite{Landsberg2012} for a more detailed discussion. 

\subsection{Segre and Veronese varieties}
The set of rank-$1$ tensors in the projective space $\Pj(\F^{n_1} \otimes \F^{n_2} \otimes \cdots \otimes \F^{n_d})$ is a projective variety, called the \emph{Segre variety}. It is the image of the \emph{Segre map}
\begin{align*}
 \mathrm{Seg} : \Pj\F^{n_1} \times \Pj\F^{n_2} \times \cdots \times \Pj\F^{n_d} &\to \Pj(\F^{n_1} \otimes \F^{n_2} \otimes \cdots \otimes \F^{n_d}) \cong \Pj\F^{n_1 n_2 \cdots n_d} \\
 ([\sten{a}{ }{1}], [\sten{a}{ }{2}], \ldots, [\sten{a}{ }{d}]) &\mapsto [\sten{a}{ }{1} \otimes \sten{a}{ }{2} \otimes \cdots \otimes \sten{a}{ }{d}]
\end{align*}
where $[\vect{a}] = \{ \lambda \vect{a} \;|\; \lambda\in\F_0 \}$ is the equivalence class of $\vect{a} \in \F^n \setminus \{0\}$. The Segre variety will be denoted by $\Var{S}$. Its dimension is 
\(
 \Sigma = \dim \Var{S} = \sum_{k=1}^d (n_k - 1).
\)

The symmetric rank-$1$ tensors in $\Pj(\F^{n} \otimes \cdots \otimes \F^{n})$ constitute an algebraic variety that is called the \emph{Veronese variety}. It is obtained as the image of
\begin{align*}
 \mathrm{Ver} : \Pj\F^n \to \Pj(\F^n \otimes \cdots \otimes \F^n),\;
 [\vect{a}] \mapsto [\vect{a}^{\otimes d}].
\end{align*}
The Veronese variety will be denoted by $\Var{V}$, and its dimension is $\dim \Var{V} = n-1$. The {span of the} image of the Veronese map is the projectivization of the linear subspace of $\F^n \otimes \cdots \otimes \F^n$ consisting of the \emph{symmetric tensors}, namely
\(
\{ \tensor{A} \;|\; a_{i_1,i_2,\ldots,i_d} = a_{i_{\sigma_1}, i_{\sigma_2}, \ldots, i_{\sigma_d}}, \forall \sigma \in S \},
\)
where $S$ is the set of all permutations of $\{1,2,\ldots,d\}$. This space is isomorphic to the $d$th symmetric power of $\F^n$, i.e., $S^d \F^n = \F^{\binom{n+d}{d}}$, as can be understood from 
\begin{align*}
 v_d : \Pj\F^n \to \Pj( S^d \F^n ),\;
 [\vect{a}] \mapsto [ \vect{a}^{\circ d} ] = \begin{bmatrix} a_{i_1} a_{i_2} \cdots a_{i_d} \end{bmatrix}_{ 1 \le i_1 \le i_2 \le \cdots \le i_d \le n },
\end{align*}
where $\vect{a}^{\circ d}$ is called the $d$the symmetric power of $\vect{a}$.
The homogeneous polynomials of degree $d$ in $n$ variables correspond bijectively with $S^d \F^n$ \cite{CGLM2008,IK1999}. Therefore, the elements of $\Pj(S^d \F^n)$ are often called $d$-forms or simply \emph{forms} when the degree is clear. 

\subsection{Secants of varieties}\label{sec_secant}
Define for a smooth irreducible projective variety $\Var{X} \subset \Pj V$ that is not contained in a hyperplane, such as a Segre or Veronese variety, the abstract $r$-secant variety $\operatorname{Abs} \Sec{r}{\Var{X}}$ as the closure in the Euclidean topology of
\begin{align*}
 \operatorname{Abs} \SecZ{r}{\Var{X}} := \{ ((p_1,p_2,\ldots,p_r),p) \;|\; p \in \langle p_1, p_2, \ldots, p_r \rangle,\, p_i \in \Var{X} \} \subset {\Var{X}}^{\times r}  \times \Pj V,
\end{align*}
Let the image of the projection of $\operatorname{Abs} \SecZ{r}{\Var{X}} \subset {\Var{X}}^{\times r} \times \Pj V$ onto the last factor be denoted by $\SecZ{r}{\Var{X}}$.
Then, the $r$-secant semi-algebraic set of $\Var{X}$, denoted by $\Sec{r}{\Var{X}}$, is defined as the closure in the Euclidean topology of $\SecZ{r}{\Var{X}}$. It is an irreducible semi-algebraic set because of the Tarski--Seidenberg theorem \cite{BCR1998}. For $\F=\C$, the Zariski-closure coincides with the Euclidean closure and $\Sec{r}{\Var{X}}$ is a projective variety \cite{Landsberg2012}. It follows that 
\(
 \dim \Sec{r}{\Var{X}} \le \min\{ r (\dim \Var{X} + 1) , \dim V \} - 1.
\)
If the inequality is strict then we say that $\Var{X}$ has an \emph{$r$-defective} secant semi-algebraic set. If $\Var{X}$ has no defective secant semi-algebraic sets then it is called a \emph{nondefective} semi-algebraic set. The \emph{$\Var{X}$-rank} of a point $p \in \Pj V$ is defined as the least $r$ for which $p = [p_1 + \cdots + p_r]$ with $p_i \in \widehat{\Var{X}}$; we will write $\rank{p} = r$. 

For a nondefective variety $\Var{X} \subset \Pj V$ not contained in a hyperplane, we define the \emph{expected smallest typical} \emph{rank} of $\Var{X}$ as the least integer larger than
\[
 \overline{r}_\Var{X} = \frac{\dim V}{ 1 + \dim \Var{X} },
\]
namely $\lceil \overline{r}_\Var{X} \rceil$. With this definition, the expected smallest typical rank of a nondefective complex Segre variety $\Var{S}_\C \subset \Pj V$ coincides with the value of $r$ for which $\Sec{r}{\Var{S}_\C} = \Pj V$, so that $\SecZ{\lceil\overline{r}_{\Var{S}_\C}\rceil}{\Var{S}_\C}$ is a Euclidean-dense subset of $\Pj V$. In the case of a nondefective real Segre variety $\Var{S}_\R \subset \Pj V$, the expected smallest typical rank as defined above coincides with the smallest typical rank; a rank $r$ is \emph{typical} if the affine cone over $\SecZ{r}{\Var{S}_\R} \subset \Pj V$ is open in the Euclidean topology on $V$.
\footnote{In the case of $\F=\R$, the inclusion $\Sec{\lceil\overline{r}_{\Var{S}_\R}\rceil}{\Var{S}_\R} \subset \Pj V$ can be strict, as the closures in the Euclidean and Zariski topologies can be different, leading to several typical ranks, see, e.g., \cite{BBO2015,Blekherman2015,Comon2012,Landsberg2012}. It is nevertheless still true that the closure in the Zariski topology of $\Sec{r}{\Var{S}_\R}$ is $\Pj V$ for every typical rank $r$.}
Note that $\overline{r}_{\Var{S}_\R} = \overline{r}_{\Var{S}_\C}$ and that in $\F=\C$ there is only one typical rank, which is hence the generic rank.
For a nondefective variety $\Var{X}$, the generic element $[p] \in \Sec{r}{\Var{X}}$ with $r \le \overline{r}_\Var{X}$ has $\rank{[p]} = \rank{p} = r$, and, furthermore, it admits finitely many decompositions of the form 
\(
p = p_1 + \cdots + p_r
\)
with \(p_i \in \widehat{\Var{X}}\).
For $r > \overline{r}_\Var{X}$, it follows from a dimension count that the generic $[p] \in \Sec{r}{\Var{X}}$ admits infinitely many decompositions of the foregoing type, because the generic fiber of the projection map $\operatorname{Abs}\Sec{r}{\Var{X}} \to \Sec{r}{\Var{X}}$ has dimension $r(\dim \Var{X} + 1) - \dim V$. {This observation also holds for $r$-defective secant semi-algebraic sets where the generic fiber has a dimension equal to the defect.}

\subsection{Inclusions, projections, and flattenings}
Let $\vect{h} \party \vect{k} = \{1,2,\ldots,d\}$ with $\vect{h}$ and $\vect{k}$ of cardinality $s>0$ and $t>0$ respectively. Several criteria for identifiability rely on the natural inclusion into two-factor Segre varieties, namely
\[
 \Var{S} = \operatorname{Seg}(\Pj\F^{n_1} \times \cdots \times \Pj\F^{n_d}) \hookrightarrow \operatorname{Seg}\bigl( \Pj(\F^{n_{h_1}} \otimes \cdots \otimes \F^{n_{h_s}}) \times \Pj( \F^{n_{k_1}} \otimes \cdots \otimes \F^{n_{k_t}} ) \bigr) = \Var{S}_{\vect{h},\vect{k}},
\]
or the inclusion into three-factor Segre varieties, which can be defined analogously and for which we employ the notation $\Var{S}_{\vect{h},\vect{k},\vect{l}}$, where $\vect{h} \party \vect{k} \party \vect{l} = \{1,2,\ldots,d\}$.

Let $\vect{h}\subset\{1,2,\ldots,d\}$ be of cardinality $s>0$.
Define the projections 
\begin{align*}
 \pi_\vect{h} : \Var{S} = \operatorname{Seg}( \Pj\F^{n_1} \times \Pj\F^{n_2} \times \cdots \times \Pj\F^{n_d} ) &\to \operatorname{Seg}( \Pj\F^{n_{h_1}} \times \cdots \times \Pj\F^{n_{h_{s}}}) = \Var{S}_\vect{h} \\
 [\vect{a}_{1} \otimes \vect{a}_{2} \otimes \cdots \otimes \vect{a}_{d}] &\mapsto [\vect{a}_{h_1} \otimes \vect{a}_{h_2} \otimes \cdots \otimes \vect{a}_{h_{s}}].
\end{align*}
This definition can be extended to every rank-$r$ tensor in $\Sec{r}{\Var{S}}$ through linearity. 
We will abuse notation by writing $\pi_{\vect{h}}( p ) = \vect{a}_{h_1} \otimes \cdots \otimes \vect{a}_{h_{s}}$ if $p = \vect{a}_1 \otimes \cdots \otimes \vect{a}_d \in \widehat{\Var{S}}$.

Flattenings are defined as follows. Let $\vect{h} \party \vect{k} = \{1,2,\ldots,d\}$ with $\vect{h}$ and $\vect{k}$ of cardinality $s\ge0$ and $t\ge0$ respectively. Then, the $(\vect{h},\vect{k})$-flattening, or simply $\vect{h}$-flattening, of $p \in \widehat{\Var{S}}$ is the natural inclusion of $p \in \widehat{\Var{S}}$ into $\widehat{\Var{S}}_{\vect{h},\vect{k}}$:
\[
 p_{(\vect{h})} = \pi_{\vect{h}}(p) \pi_{\vect{k}}(p)^T \in \widehat{\Var{S}}_{\vect{h},\vect{k}} \subset \F^{n_{h_1} \cdots n_{h_{s}}} \otimes \F^{n_{k_1} \cdots n_{k_{t}}} \cong \F^{n_{h_1} \cdots n_{h_{s}} \times n_{k_1} \cdots n_{k_{t}}}.
\]
By convention $\pi_{\emptyset}(p) = 1$. 
A $(\vect{h},\vect{k})$-flattening of a rank-$r$ tensor is obtained by extending the above definition through linearity. 

\section{Generic identifiability of tensors and forms}\label{sec_identifiability}
{Let $\Var{X}$ be an irreducible algebraic $\F$-variety. By convention, we call a rank-$r$ decomposition $p = p_1 + \cdots + p_r$, $p_i \in \widehat{\Var{X}}$, distinct from another decomposition $p = q_1 + \cdots + q_r$, $q_i \in \widehat{\Var{X}}$, if there does not exist a permutation $\sigma$ of $\{1,2,\ldots,r\}$ such that $p_i = q_{\sigma_i}$ for all $i$. 
We say that $\Var{X}$ is \emph{generically $r$-identifiable} if the set of tensors with multiple distinct \emph{complex} decompositions in $\Sec{r}{\Var{X}}$ is contained in a proper Zariski-closed subset of $\Sec{r}{\Var{X}}$.} This concept is meaningful only when $r$ is \emph{subtypical}, i.e., $r < \overline{r}_\Var{X}$, or if the tensor space is \emph{perfect}, so that $r = \overline{r}_\Var{X}$ is an integer. The generic tensor $p \in \Sec{r}{\Var{X}}$ cannot admit a finite number of decompositions of length $r$ if $r > \overline{r}_\Var{X}$ because of the dimension argument mentioned in \refsec{sec_secant}.

The literature, specifically \cite{Mella2009,BCO2013,COV2014,COV2016,HOOS2015,GM2016}, already provides a conjecturally complete picture of complex generic $r$-identifiability of the tensor rank decomposition \refeqn{eqn_rank_decomposition} and the Waring decomposition \refeqn{eqn_waring_decomposition}.

\begin{theorem}[Chiantini, Ottaviani, and Vannieuwenhoven \cite{COV2016}] 
\label{thm_cov2016}
 Let $d \ge 3$, and let $\F = \C$ or $\R$. Let $\Var{V}_{d,n}^{\F}$ be the $d$th Veronese embedding of $\F^n$ in $\Pj S^d \F^n$. Then, $\Var{V}_{d,n}^{\F}$ is generically $r$-identifiable for all strictly subtypical ranks $r < n^{-1} \binom{n-1+d}{d}$, unless it is one of the following cases:
 \begin{enumerate}
  \item $\Var{V}_{3,6}^{\F}$ and $r = 9$; 
  \item $\Var{V}_{4,4}^{\F}$ and $r = 8$; or
  \item $\Var{V}_{6,3}^{\F}$ and $r = 9$.
 \end{enumerate}
 The generic tensor has $2$ distinct \emph{complex} decompositions in these exceptional cases.
\end{theorem}
\begin{remark}\label{uno}
This theorem was proved for $\F=\C$ in \cite{COV2016}. It can be extended to $\F=\R$ by invoking a beautiful result due to Qi, Comon, and Lim \cite{QCL2016}. 
Arguing from the abstract $r$-secant variety, the results of \cite[Section 5]{QCL2016} entail that $\Sec{r}{\Var{S}_\R}$ is not contained in the singular locus of $\Sec{r}{\Var{S}_\C}$, which has codimension at least $1$. Let the Nash stratification \cite{BCR1998} of the semi-algebraic set $\Sec{r}{\Var{S}_\R}$ be given by $\Sec{r}{\Var{S}_\R} = \cup_{i=1}^k \Var{N}_i$, with $\Var{N}_i$ a Nash manifold and $\Var{N}_i \cap \Var{N}_j = \emptyset$ if $i \ne j$. Let $\Var{N}_i$ be an arbitrary Nash manifold not contained in $\operatorname{Sing}(\Sec{r}{\Var{S}_\C})$, and let $[p] \in \Var{N}_i \setminus \operatorname{Sing}(\Sec{r}{\Var{S}_\C})$ be generic so it has a decomposition $p = p_1 + \cdots + p_r$, where $p_i \in \widehat{\Var{S}}_\R$. Since $p$ is smooth in $\Sec{r}{\Var{S}_\C}$ its tangent space is $\Tang{p}{\Sec{r}{\Var{S}_\C}} = \langle \Tang{p_1}{\Var{S}_\C}, \ldots, \Tang{p_r}{\Var{S}_\C} \rangle = \langle \Tang{p_1}{\Var{S}_\R} \otimes \C, \ldots, \Tang{p_r}{\Var{S}_\R}\otimes\C \rangle = \Tang{p}{\Var{S}_\R} \otimes \C$ by Terracini's lemma. Hence, as complex varieties, $\dim \Var{N}_i = \dim \Sec{r}{\Var{S}_\C}$. Let $\Var{U} \subset \Sec{r}{\Var{S}_\C}$ be the locus where complex $r$-identifiability fails, which is of codimension at least $1$ by Theorem 1.1 of \cite{COV2016}. Then $\Var{N}_i \cap \Var{U}$ is contained in a proper Zariski-closed set of $\Var{N}_i \subset \Sec{r}{\Var{S}_\C}$, namely $\Var{U}$,
proving that $\Sec{r}{\Var{S}_\R}$ is generically $r$-identifiable if $\Sec{r}{\Var{S}_\C}$ is generically $r$-identifiable.

If complex $r$-identifiability fails on a Zariski-open set, {then there is a Euclidean-open set of real decomposition of real rank $r$ admitting multiple complex decompositions}, however we do not presently know how many of these are \emph{real}. We leave this as an open problem warranting further research.
\end{remark}

This theorem completely settles the question concerning the number of complex Waring decompositions \refeqn{eqn_waring_decomposition} of the generic symmetric tensor of strictly subtypical rank $r < \overline{r}_\Var{V}$: aside from the listed exceptions, it is one. In the perfect case where $\overline{r}_\Var{V}$ is an integer and $\F=\C$, the following remarkable result was recently proved.
\begin{theorem}[Galuppi and Mella \cite{GM2016}] \label{thm_perfect}
 Let $d \ge 3$. Let $\Var{V}_{d,n}$ be the $d$th Veronese embedding of $\C^n$ in $\Pj S^d \C^n$ and assume that $\overline{r}_\Var{V} = n^{-1} \binom{n-1+d}{d}$ is an integer. {$\Var{V}_{d,n}$ is generically $\overline{r}_{\Var{V}_{d,n}}$-identifiable if and only if it is either $\Var{V}_{2k+1,2}$ with $k\ge1$, $\Var{V}_{3,4}$ or $\Var{V}_{5,3}$.}
\end{theorem}

In summary we can state that the generic symmetric tensor k {of rank $r$} in all but a few tensor spaces $S^d \F^n$ admits a unique Waring decomposition over $\F$ if $r$ is subtypical, while it is expected to admit several decompositions if {$r \ge \overline{r}_{\Var{V}_{d,n}}$}.

The theory of generic identifiability of the Segre variety is less developed than the Veronese variety. Because of the corroborating evidence in \cite{CO2012,BCO2013,COV2014,DDL2015,HOOS2015,Strassen1983}, the following conjectures are believed to be true.
\begin{conjecture}
\label{conj_idf}
 Let $d \ge 3$, and let $n_1 \ge \cdots \ge n_d \ge 2$. Let $\Var{S}_\F = \operatorname{Seg}(\Pj\F^{n_1} \times \cdots \times \Pj\F^{n_d})$ be the Segre variety in $\Pj(\F^{n_1} \otimes \cdots \otimes \F^{n_d})$. Then, $\Var{S}_\F$ is generically $r$-identifiable for all strictly subtypical ranks $r < \overline{r}_{\Var{S}_\F}$, unless it is one of the 
 following cases:
 \begin{enumerate}
  \item $n_1 > \prod_{k=2}^d n_k - \sum_{k=2}^d (n_k - 1)$ and $r \ge \prod_{k=2}^d n_k - \sum_{k=2}^d (n_k - 1)$;
  \item $\Var{S} = \operatorname{Seg}(\Pj\F^{4} \times \Pj\F^{4} \times \Pj\F^3)$ and $r = 5$;
  \item $\Var{S} = \operatorname{Seg}(\Pj\F^{n} \times \Pj\F^n \times \Pj\F^2 \times \Pj\F^2)$ and $r = 2n-1$;
  \item $\Var{S} = \operatorname{Seg}(\Pj\F^{4} \times \Pj\F^{4} \times \Pj\F^4)$ and $r = 6$;
  \item $\Var{S} = \operatorname{Seg}(\Pj\F^{6} \times \Pj\F^{6} \times \Pj\F^3)$ and $r = 8$; or
  \item $\Var{S} = \operatorname{Seg}(\Pj\F^{2} \times \Pj\F^{2} \times \Pj\F^2 \times \Pj\F^2 \times \Pj\F^2)$ and $r = 5$;
 \end{enumerate}
 The first three cases generically admit {infinitely many} decompositions \cite{AOP2009,BCO2013}. Case (4) generically admits $2$ \emph{complex} decompositions \cite{CO2012}, case (5) is expected\footnote{This statement is true with probability $1$ due to \cite[Proposition 4.1]{HOOS2015} and \cite[Section 5.1]{HR2015}.} to generically admit $6$ \emph{complex} decompositions, and case (6) admits generically $2$ \emph{complex} decompositions \cite{BC2013}.
\end{conjecture}
\begin{remark}
The conjecture was initially stated for $\F=\C$ in \cite{BCO2013,COV2014}. Theorem 1.1 of \cite{COV2014}, which proves \refconj{conj_idf} for all $n_1n_2\cdots n_d \le 15000$ with $\F=\C$, can be extended to $\F=\R$ as in \refrem{uno} by invoking Qi, Comon, and Lim's analysis \cite{QCL2016}. 
\end{remark}

\begin{conjecture}[Hauenstein, Oeding, Ottaviani, and Sommese \cite{HOOS2015}]
Let $d \ge 3$, and let $n_1 \ge n_2 \ge \cdots \ge n_d \ge 2$. Let $\Var{S} = \operatorname{Seg}(\Pj\C^{n_1} \times \Pj\C^{n_2} \times \cdots \times \Pj\C^{n_d})$ be the Segre variety in $\Pj(\C^{n_1} \otimes \C^{n_2} \otimes \cdots \otimes \C^{n_d})$, and assume that $\overline{r}_\Var{S}$ is an integer. Then, $\Var{S}$ is \emph{not} generically $\overline{r}_\Var{S}$-identifiable, unless it is one of the following cases:
\begin{enumerate}
  \item $\Var{S} = \operatorname{Seg}(\Pj\C^{5} \times \Pj\C^{4} \times \Pj\C^3)$, or
  \item $\Var{S} = \operatorname{Seg}(\Pj\C^{3} \times \Pj\C^2 \times \Pj\C^2 \times \Pj\C^2)$.
\end{enumerate}
\end{conjecture}

\section{An effective criterion for specific identifiability}\label{sec_reshaped_kruskal}
We formalize the concept of an effective criterion for specific identifiability.

\begin{definition}
Let $\Var{X} \subset \Pj V$ be a generically $r$-identifiable variety. A criterion for specific $r$-identifiability of $\Var{X}$ is called \emph{effective} if it certifies identifiability on a dense subset of $\Sec{r}{\Var{X}}$ in the Euclidean topology. 
\end{definition}

Thus, if we consider a probability distribution with noncompact support on the affine cone of a generically $r$-identifiable variety $\Var{X}$, then the probability that an effective criterion for specific $r$-identifiability fails to certify identifiability of $p = p_1 + \cdots + p_r$ is zero when the $p_i$'s were randomly sampled from the probability distribution on $\widehat{\Var{X}}$. 

\subsection{The reshaped Kruskal criterion}
We show that Kruskal's criterion \cite{Kruskal1977} is effective when combined with reshaping. 
The key to this criterion is the notion of \emph{general linear position} (GLP) \cite{Landsberg2012}. A set of points $S = \{p_1, p_2, \ldots, p_r\} \subset \Pj V$ is in GLP if for $s = \min\{r, \dim V\}$, the subspace spanned by every subset $R \subset S$ of cardinality $s$ is of the maximal dimension $s-1$.
This means that no 2 points coincide, no $3$ points are on a line, no $4$ points are on a plane, and so forth. 
The \emph{Kruskal rank} of a finite set of points $S \subset \Pj V$ is then the largest value $\kappa$ for which every subset of $\kappa$ points of $S$ is in GLP.

Let $p_i = \sten{a}{i}{1} \otimes \cdots \otimes \sten{a}{i}{d}$, $i=1,\ldots,r$, be a collection of $r$ points in $\widehat{\Var{S}}$. Then we denote the \emph{factor matrices} of the points $p_i$ by 
\[
 A_k 
= \begin{bmatrix} \sten{a}{1}{k} & \sten{a}{2}{k} & \cdots & \sten{a}{r}{k} \end{bmatrix} 
= \begin{bmatrix} \pi_{\{k\}}(p_1) & \pi_{\{k\}}(p_2) & \cdots & \pi_{\{k\}}(p_r) \end{bmatrix} \in \F^{n_k \times r}
\]
for $k=1,2,\ldots,d$.
Letting $\vect{h} \subset \{1,2,\ldots,d\}$ be an ordered set, we define for brevity
\[
A_{\vect{h}} 
= A_{h_1} \odot A_{h_2} \odot \cdots \odot A_{h_{|\vect{h}|}}
= \begin{bmatrix} \pi_{\vect{h}}(p_1) & \pi_{\vect{h}}(p_2) & \cdots & \pi_{\vect{h}}(p_r) \end{bmatrix}.
\]
Kruskal's criterion for specific identifiability may then be formulated as follows.

\begin{proposition}[Kruskal's criterion \cite{Kruskal1977}] \label{prop_kruskal}
 Let $\Var{S} = \operatorname{Seg}( \Pj\F^{n_1} \times \Pj\F^{n_2} \times \Pj\F^{n_3} )$ 
with $n_1 \ge n_2 \ge n_3 \ge 2$. Let $p \in \langle p_1, p_2, \ldots, p_r\rangle$ with $p_i = 
\sten{a}{i}{1} \otimes \cdots \otimes \sten{a}{i}{d} \in \widehat{\Var{S}}$. Let $\kappa_i$ denote the Kruskal rank of 
the factor matrices $A_1$, $A_2$ and $A_3$ respectively.
Then, $p$ is $r$-identifiable if 
\(
 r \le \frac{1}{2} (\kappa_1 + \kappa_2 + \kappa_3) - 1.
\)
 Furthermore, this criterion is effective if
 \(
  r \le \frac{1}{2}( \min\{n_1,r\} + \min\{n_2,r\} + \min\{n_3,r\} ) - 1,
 \)
 or, equivalently, letting $\delta = n_2 + n_3 - n_1 - 2$,
 \[
  r \le n_1 + \min\{ \tfrac{1}{2} \delta, \delta \};
 \]
 this is the maximum range of applicability of Kruskal's criterion.
\end{proposition}
\begin{proof}
Effectiveness was not considered in \cite{Kruskal1977}, but its proof is a consequence of 
\reflem{lem_max_glp} that will be presented shortly.
\end{proof}

\begin{remark}
While effectiveness of Kruskal's criterion is known to the experts, it is not obvious why this should have been expected. The reason is that Kruskal's criterion is not merely certifying the uniqueness of \emph{one} decomposition
\begin{align}\label{eqn_p_def}
 p = p_1 + \cdots + p_r = \sum_{i=1}^r \sten{a}{i}{1} \otimes \cdots \otimes \sten{a}{i}{d}, 
\end{align}
with $p_i \in \Var{S}_0$ and $\sten{a}{i}{k} \in \F^{n_k}$, but rather it is testing whether \emph{all} tensors 
\(p = \alpha_1 p_1 + \alpha_2 p_2 + \cdots + \alpha_r p_r,\) $\alpha_i \in \F_0$
are $r$-identifiable. 
Indeed, the Kruskal rank of a set of points is a \emph{projective} property: the Kruskal ranks of $\{[p_1],\ldots,[p_r]\}$ and $\{p_1,\ldots,p_r\}$ with $[p_i]\in\Var{S}$ are the same. This also means that Kruskal's test fails as soon as there exists one point $q = \alpha_1 p_1 + \alpha_2 p_2 + \cdots + \alpha_r p_r$, $\alpha_i \in \F_0$, that is not identifiable. Since all points $q = \alpha_1 p_1 + \alpha_2 p_2 + \cdots + \alpha_r p_r$ with some $\alpha_i=0$ are of rank at most $r-1$ and thus not $r$-identifiable, one could say that the $r$-secant plane $\langle p_1, p_2, \ldots, p_r \rangle$, $p_i \in \Var{S}_0$, is $r$-identifiable if and only if all elements of $\{ \alpha_1 p_1 + \alpha_2 p_2 + \cdots + \alpha_r p_r \;|\; \alpha_i \in \F_0 \}$ are $r$-identifiable. Kruskal's criterion is thus a criterion for checking that the $r$-secant plane $\langle p_1, p_2, \ldots, p_r \rangle$ is $r$-identifiable, when a particular tensor rank decomposition $p = p_1 + p_2 + \cdots + p_r$, $[p_i] \in \Var{S}_0$, is provided as input.

We are not aware of criteria for specific $r$-identifiability that take into account the coefficients of the given decomposition. It is not inconceivable that for some high rank $r$, the secant space $\langle p_1, \ldots, p_r \rangle$ contains both $r$-identifiable and $r$-nonidentifiable points. Perhaps taking the coefficients into account could lead to criteria for specific identifiability that apply for higher ranks.
\end{remark}

Consider a $d$-factor Segre product
\(
 \Var{S} = \operatorname{Seg}( \F^{n_1}  \times \cdots \times \F^{n_d} )
\)
and let $\vect{h} \party \vect{k} \party \vect{l} = \{1,2,\ldots,d\}$. Then,
\(
 \Var{S} = \operatorname{Seg}( \Var{S}_\vect{h} \times \Var{S}_\vect{k} \times \Var{S}_\vect{l} ) 
\hookrightarrow \Var{S}_{\vect{h},\vect{k},\vect{l}}, 
\)
so an order-$d$ rank-$1$ tensor of $\Var{S}$ can be viewed as an order-$3$ rank-$1$ tensor
 in $\Var{S}_{\vect{h},\vect{k},\vect{l}}$. 
We could try to apply Kruskal's criterion by interpreting $p\in\SecZ{r}{\Var{S}}$ as a third-order tensor 
$p\in\SecZ{r}{\Var{S}_{\vect{h},\vect{k},\vect{l}}}$.
Note that $\Sec{r}{\Var{S}}$ is a Zariski-closed subset\footnote{We are assuming here that 
$\Var{S}_{\vect{h},\vect{k},\vect{l}}$ is nondefective \cite{AOP2009}.} of $\Sec{r}{\Var{S}_{\vect{h},\vect{k},\vect{l}}}$ 
so that \emph{we cannot immediately conclude from \refprop{prop_kruskal} that Kruskal's criterion applied to reshaped 
tensors is effective}. The range of effectiveness follows from the following result.

\begin{lemma}\label{lem_max_glp}
Let $\Var{S} = \operatorname{Seg}( \Pj\F^{n_1} \times \cdots \times \Pj\F^{n_d} )$ with $\F = \R$ or $\C$. Then, there exists a Euclidean-dense, Zariski-open subset $G \subset \Var{S}^{\times r}$ such that for every nonempty $\vect{h} \subset \{1,2,\ldots,d\}$ and every $(p_1,p_2,\ldots,p_r) \in G$, the points $(\pi_{\vect{h}}(p_1),\pi_{\vect{h}}(p_2),\ldots,\pi_{\vect{h}}(p_r)) \in \Var{S}_\vect{h}$ are in GLP.
\end{lemma}
\begin{proof}
For $r=1$ the statement is obvious. So assume that $r \ge 2$.

We prove the existence of $G = G_{\{1,2,\ldots,d\}}$ by induction on the cardinality of $\vect{h}$. Specifically, we show that for every $\vect{h} \subset \{1,2,\ldots,d\}$ the configurations $(p_1,\ldots,p_r) \in \Var{S}_\vect{h}$ that are not in GLP form a Zariski-closed subset $G_\vect{h} \subset \Var{S}_\vect{h}^{\times r}$.
Let $\vect{h} = \{i\}$. Then, $\Var{S}_{\vect{h}} = \Pj \F^{n_i}$. Let $s = \min\{n_i, r\}$. By definition, the configurations in $\Var{S}_\vect{h}^{\times r}$ wherein the first set of $s$ points are not in GLP can be described as
\begin{align}\label{eqn_closed_set_1}
\bigcup_{[q_2], \ldots, [q_r] \in \Var{S}_\vect{h}} \bigcup_{\alpha_2,\ldots,\alpha_s \in \F} ( [\alpha_2 q_2 + \cdots + \alpha_s q_s], [q_2], \ldots, [q_r] ) \subset \Var{S}_\vect{h}^{\times r}, 
\end{align}
which can be obtained from a projection of $\Pj\F^{s-1} \times \Var{S}_\vect{h}^{\times r-1}$, so that its dimension is strictly less than $\dim \Var{S}_\vect{h}^{\times r}$ because $\min\{r,n_i\}-2 = \dim \Pj\F^{s-1} < \dim \Var{S}_\vect{h} = n_i - 1.$
Hence \refeqn{eqn_closed_set_1} is a Zariski-closed set in $\Var{S}_\vect{h}^{\times r}$. 
The configurations in $\Var{S}_\vect{h}^{\times r}$ where $q_i \in \Var{S}_\vect{h}$ is a linear combination of $s-1$ other points in $\Var{S}_\vect{h}$ can all be obtained from permuting the factors in the Cartesian product in \refeqn{eqn_closed_set_1}. It follows that the union of all these Zariski-closed sets is precisely the Zariski-closed subset $G_\vect{h} \subset \Var{S}_\vect{h}^{\times r}$ of configurations $(q_1,\ldots,q_r) \in \Var{S}_\vect{h}^{\times r}$ that are not in GLP. Note that the sets $G_\vect{h}$ are $\F$-varieties because linear dependence of vectors can be formulated as a collection of determinantal equations with coefficients in $\mathbb{Z} \subset \F$.

Assume now that the statement is true for all $\vect{j} \subset \{1,2,\ldots,d\}$ whose cardinality is less than or equal to $k-1$. Then, we prove that it is true for every $\vect{h}\subset \{1,2,\ldots,d\}$ of cardinality $k$. Let $s = \min\{ \prod_{i\in\vect{h} } n_i, r\}$. By induction, the sets $G_\vect{j}$ with $\vect{j} \subsetneq \vect{h}$ are Zariski-closed. Consider the surjective map
\begin{align*}
 (\Var{S}_\vect{j}^{\times r} \setminus G_\vect{j}) \times (\Var{S}_{\vect{h}\setminus\vect{j}}^{\times r} \setminus G_{\vect{h}\setminus\vect{j}}) &\to \Var{S}_\vect{h}^{\times r} \setminus H_{\vect{h},\vect{j}} \\
 ([x_1],[x_2],\ldots,[x_r]) \times ([y_1],[y_2],\ldots,[y_r]) &\mapsto ([x_1 \otimes y_1],[x_2 \otimes y_2],\ldots,[x_r \otimes y_r]),
\end{align*}
where $H_{\vect{h},\vect{j}}$ can be defined as
\[
 H_{\vect{h},\vect{j}} =  \{ ([x_1 \otimes y_1],\ldots,[x_r \otimes y_r]) \;|\; ([x_1],\ldots,[x_r]) \in G_\vect{j} \text{ or } ([y_1],\ldots,[y_r]) \in G_{\vect{h}\setminus\vect{j}} \}.
\]
Let $\Pi_\vect{l} = \prod_{i\in\vect{l}} n_i$ for any $\vect{l} \subset \{1,2,\ldots,d\}$. Let a Gr\"obner basis of the ideal of $G_\vect{j}$ consist of the $\F$-polynomials 
$f_i(x_{1,1},x_{2,1},\ldots,x_{\Pi_\vect{j},1},\ldots,x_{1,r},x_{2,r},\ldots,x_{\Pi_\vect{j},r}),$ 
and similarly let 
$g_i(y_{1,1},y_{2,1},\ldots,y_{\Pi_{\vect{h}\setminus\vect{j}},1},\ldots,y_{1,r},y_{2,r},\ldots,y_{\Pi_{\vect{h}\setminus\vect{j}},r})$
be the polynomials in a Gr\"obner basis of the ideal of $G_{\vect{h}\setminus\vect{j}}$.
Let $Z_{i,j,k} = x_{i,k}y_{j,k}$ with $i=1,\ldots,\Pi_{\vect{j}}$, $j=1,\ldots,\Pi_{\vect{h}\setminus\vect{j}}$, and $k=1,\ldots,r$ be variables for $\Var{S}_\vect{h}^{\times r}$. 
Then, $H_{\vect{h},\vect{j}} \subset \Var{S}_\vect{h}^{\times r}$ is contained in the variety whose ideal is spanned by the following set of $\F$-polynomials:
\begin{multline*}
f_i( Z_{1,\mu,1},\ldots,Z_{\Pi_\vect{j},\mu,1}, \ldots, Z_{1,\mu,r},\ldots,Z_{\Pi_\vect{j},\mu,r} ) 
\cdot \phantom{0} \\ g_j( Z_{\nu,1,1}, \ldots, Z_{\nu,\Pi_{\vect{h}\setminus\vect{j}},1},\ldots, Z_{\nu,1,r}, \ldots, Z_{\nu,\Pi_{\vect{h}\setminus\vect{j}},r})
\end{multline*}
for every $(i,j)$, $\mu = 1, 2, \ldots, \Pi_{\vect{h}\setminus\vect{j}}$, and $\nu = 1,2,\ldots,\Pi_\vect{j}$. As $G_\vect{j}$ is Zariski-closed by induction, $H_{\vect{h},\vect{j}}$ is Zariski-closed. Thus the finite union 
\(
H_\vect{h} = \bigcup_{\vect{j}\subsetneq\vect{h}} H_{\vect{h},\vect{j}}
\)
is a Zariski-closed set. Now, $\Var{S}_\vect{h}^{\times r} \setminus H_\vect{h}$ contains all configurations $(p_1,p_2,\ldots,p_r)$ for which for every $\vect{j}\subsetneq\vect{h}$ we have that $(\pi_\vect{j}(p_1), \pi_\vect{j}(p_2),\ldots,\pi_\vect{j}(p_r))$ is in GLP. As in the proof of the base case, it is straightforward to show that there exists a Zariski-closed set $G_\vect{h}' \subset \Var{S}_\vect{h}^{\times r}$ that contains all configurations that are not in GLP. The proof is then concluded by setting $G_\vect{h} = G_\vect{h}' \cup H_\vect{h}$. 
\end{proof}

The foregoing result has some implications for the Khatri--Rao product that could be of independent interest, generalizing \cite[Corollary 1]{JStB2001} to the real case.
\begin{corollary}
 Let $(A_1,A_2,\ldots,A_d) \in \F^{n_1 \times r} \times \F^{n_2 \times r} \times \cdots \times \F^{n_d \times r}$ be generic. Then, for every $\vect{h} \subset \{1,2,\ldots,d\}$ of cardinality $k>0$ the matrix
 \(
  A_{h_1} \odot A_{h_2} \odot \cdots \odot A_{h_{k}}
 \)
 has the maximal rank, i.e., $\min\{ r, \prod_{i\in\vect{h}} n_i \}$.
\end{corollary}

It follows immediately from \refprop{prop_kruskal} and \reflem{lem_max_glp} that Kruskal's theorem with reshaping is effective in the broadest range that one could have expected.
\begin{theorem}[Reshaped Kruskal criterion]\label{thm_reshaped_kruskal}
Let $d \ge 3$, and 
let 
\(
\Var{S} = \operatorname{Seg}( \Pj\F^{n_1} \times \Pj\F^{n_2} \times \cdots \times \Pj\F^{n_d} ),
\)
and let $p \in \langle p_1, p_2, \ldots, p_r\rangle$ with $p_i = \sten{a}{i}{1}\otimes\cdots\otimes\sten{a}{i}{d} \in \widehat{\Var{S}}$.
Let $\Pi_\vect{m} = \prod_{\ell \in \vect{m}} n_\ell$ for any $\vect{m} \subset \{1,2,\ldots,d\}$.
Let $\vect{h} \party \vect{k} \party \vect{l} = \{1,2,\ldots,d\}$ be such that 
$\Pi_{\vect{h}} \ge \Pi_{\vect{k}} \ge \Pi_\vect{l}$. 
Let the Kruskal ranks of the factor matrices $A_{\vect{h}}$, $A_{\vect{k}}$ and $A_{\vect{l}}$ 
be denoted by $\kappa_1$, $\kappa_2$ and $\kappa_3$ respectively. Then, $p$ is $r$-identifiable if 
 \(
  r \le \frac{1}{2} (\kappa_1 + \kappa_2 + \kappa_3) - 1.
 \)
 Furthermore, letting $\delta = \Pi_\vect{k} + \Pi_\vect{l} - \Pi_\vect{h} - 2$, this criterion is effective if
 \begin{align}\label{eqn_r_bound}
  r \le \Pi_\vect{h} + \min\{ \tfrac{1}{2} \delta, \delta \}.
 \end{align}
\end{theorem}

\begin{romanex}
Let us consider a rank-$18$ tensor in $\R^{6\times5\times4\times3\times2}$ whose factor matrices $A_k$ were generated 
in Macaulay2 \cite{M2} as follows:
\begin{verbatim}
n = {6,5,4,3,2};
for i from 1 to length(n) do (
  A_i = matrix apply(n_i, jj->apply(r, kk->random(-99,99)));
);
\end{verbatim}
Let $\sten{a}{i}{k}$ denote the $k$th column of $A_k$, $k=1,\ldots,5$. Then these factor matrices naturally represent 
the tensor 
\(
\tensor{A} = \sum_{i=1}^{18} \sten{a}{i}{1} \otimes \cdots \otimes \sten{a}{i}{5}.
\)
One could try applying the higher-order version of Kruskal's theorem due to Sidiropoulos and Bro \cite{SB2000}, which 
states that $\tensor{A}$'s decomposition is unique if $2r \le \kappa_1 + \cdots + \kappa_5 - 4$, where $\kappa_i$ is the Kruskal rank of $A_k$. We have
\begin{verbatim}
apply(length(n),i->kruskalRank(A_(i+1)))
o1 = {6, 5, 4, 3, 2}
\end{verbatim}
Herein, the function \texttt{kruskalRank} in the ancillary file \texttt{reshapedKruskal.m2} computes the Kruskal rank of the input matrix. Since $36 \not\le 6+5+4+3+2-4 = 16$, an application of the higher-order Kruskal criterion is inconclusive. Let us instead consider $\tensor{A}$ as an element of $(\R^{5}\otimes\R^{4}) \otimes (\R^{6} \otimes \R^{3}) \otimes \R^2$ by permuting and grouping the factors in the tensor product. The factor matrices of $\tensor{A}$ in this interpretation are $A_2 \odot A_3 \in \R^{20 \times 18}$, $A_1 \odot A_4 \in \R^{18 \times 18}$ and $A_5 \in \R^{2 \times 18}$. The Kruskal ranks of these matrices can be computed by employing \texttt{reshapedKruskal.m2} as follows:
\begin{verbatim}
{kruskalRank(kr(A,{2,3})), kruskalRank(kr(A,{1,4})), kruskalRank(A_5)}
o2 = {18, 18, 2}
\end{verbatim}
The \texttt{kr(A,L)} function computes the Khatri--Rao product of the $A_{L_i}$, which are all matrices, and where $L$ is a list of indices; for example, \texttt{kr(A,\{i,j\})} computes $A_i \odot A_j$.
Applying Kruskal's criterion to this tensor, we find $36 \le 18 + 18 + 2 - 2 = 36$, so that $\tensor{A}$ is $18$-identifiable as an element of $\R^{20}\times\R^{18}\times\R^2$. It follows that $\tensor{A}$ is also $18$-identifiable in the original space. 

As is stated in \refthm{thm_reshaped_kruskal}, the reshaping of the tensor influences the effective range in which the reshaped Kruskal criterion applies. For instance, if we had considered $\tensor{A}$ as an element of $(\R^{6}\otimes\R^{5}) \otimes (\R^{4} \otimes \R^{3}) \otimes \R^2$, then the Kruskal ranks of $A_1 \odot A_2$, $A_3 \odot A_4$ and $A_5$ are determined by Macaulay2 to be 
\begin{verbatim}
{kruskalRank(kr(A,{1,2})), kruskalRank(kr(A,{3,4})), kruskalRank(A_5)}
o3 = {18, 12, 2}
\end{verbatim}
With this reshaping $36 \not\le 18+12+2-2 = 30$, so that the test is inconclusive.
\end{romanex}

\subsection{A heuristic for reshaping}\label{sec_heurisic_reshaping}
Choosing the partition $\vect{h} \party \vect{k} \party \vect{l}$ in \refthm{thm_reshaped_kruskal} influences the range in which the criterion is effective. Note that if $\Pi_\vect{h} \ge r \ge \Pi_\vect{k} \ge \Pi_\vect{l}$, then the criterion in \refthm{thm_reshaped_kruskal} is effective for $r \le \Pi_\vect{k} + \Pi_\vect{l} - 2$. After our discussions with I.~Domanov, we realized that a good heuristic yielding a large effective range of identifiability consists of first choosing
\[
\vect{k} \in \underset{\substack{\vect{y}\subset\{1,\ldots,d\},\\\vect{x}\party\vect{y}\party\vect{z} = \{1,\ldots,d\},\\\Pi_\vect{x} \ge \Pi_\vect{y} \ge \Pi_\vect{z}}}{\arg\max} \Pi_\vect{y},
\quad\text{and then}\quad
\vect{h} \in \underset{\substack{\vect{x}\subset\{1,\ldots,d\},\\\vect{x}\party\vect{k}\party\vect{z} = \{1,\ldots,d\},\\\Pi_\vect{x} \ge \Pi_\vect{k} \ge \Pi_\vect{z}}}{\arg\min} \Pi_\vect{x},
\]
and finally $\vect{l} = \{1,2,\ldots,d\}\setminus(\vect{h} \cup \vect{k})$. One should thus first try to maximize the second-largest reshaped dimension $\Pi_\vect{k}$, and then minimize the largest reshaped dimension.

\begin{example}
Let $d = 4$. Then there are $6$ distinct partitions of $\{1,2,3,4\}$, namely
\(\sigma_{1,2} = \{1,2\} \party \{3\} \party \{4\},\)
\(\sigma_{1,3} = \{1,3\} \party \{2\} \party \{4\},\)
\(\sigma_{1,4} = \{1,4\} \party \{2\} \party \{3\},\)
\(\sigma_{2,3} = \{2,3\} \party \{1\} \party \{4\},\)
\(\sigma_{2,4} = \{2,4\} \party \{1\} \party \{3\},\) and
\(\sigma_{3,4} = \{3,4\} \party \{1\} \party \{2\}.\)
The effective range of the reshaped Kruskal criterion in \refthm{thm_reshaped_kruskal} corresponding to these partitions is given below for a few arbitrarily chosen shapes:
\begin{align*}\small
\begin{array}{ccccccc}
\toprule
(n_1,n_2,n_3,n_4)  & \sigma_{1,2} & \sigma_{1,3} & \sigma_{2,3} & \sigma_{1,4} &  \sigma_{2,4} & \sigma_{3,4} \\
\midrule
(17, 13, 13, 2)  & 13 & 13 & 17 & 24 & \mathbf{27} & \mathbf{27} \\
(17, 8, 3, 2)    &  3 &  8 & \mathbf{17} &  9 & 17 & 12 \\
(15, 15, 11, 10) & 19 & 23 & 23 & 24 & 24 & \mathbf{28} \\
(15, 13, 9, 4)   & 11 & 15 & 17 & 20 & 22 & \mathbf{26} \\
(12, 10, 7, 7)   & 12 & 15 & 17 & 15 & 17 & \mathbf{20} \\
\bottomrule
\end{array}
\end{align*}
The values highlighted in bold correspond to the choice of the heuristic. In all of these examples, the heuristic choice resulted in the largest range for which the reshaped Kruskal criterion could be applied. 
\end{example}

The heuristic is asymptotically optimal in two extreme cases, namely when $\Var{S}$ is unbalanced and when $n_1 = n_2 = \cdots = n_d = n$. We expect that the proposed partitioning should perform reasonably well in other instances as well.

\begin{proposition}\label{prop_best_range}
Let $\Var{S} = \operatorname{Seg}(\Pj\F^n \times \cdots \times \Pj\F^n)$ be a $d$-factor Segre product. Then the reshaped Kruskal criterion is effective for
\[
r \le 
\begin{cases}
\tfrac{3}{2} n - 1 &\text{if } d = 3, \\
2n - 2 &\text{if } d = 4, \\
n^{\lfloor(d-1)/2\rfloor} + \frac{1}{2}n^{d - 2\lfloor(d-1)/2\rfloor} - 1 &\text{if } d \ge 5.
\end{cases}
\]
Furthermore, for large $n$ this is the largest range in which \refthm{thm_reshaped_kruskal} applies.
\end{proposition}
\begin{proof}
The case $d=3$ is \refprop{prop_kruskal}.
 
In the case $d=4$, the only admissible reshaping, up to a permutation of the factors, is to a $n^2 \times n \times n$ tensor. An application of \refthm{thm_reshaped_kruskal} yields the result. Since it is the only admissible reshaping, it is optimal.

Let $d \ge 5$.
Let the cardinality of $\vect{h}$, $\vect{k}$, and $\vect{l}$ be respectively $\alpha$, $\beta$, and $\gamma$, where $\alpha+\beta+\gamma=d$ and $\alpha\ge\beta\ge\gamma\ge1$. Suppose first that $r \ge n^\alpha \ge n^\beta \ge n^\gamma$, so that the criterion is effective if $n^\alpha \le r \le \tfrac{1}{2}(n^\alpha + n^\beta + n^\gamma) - 1$. For sufficiently large $n$, these inequalities are consistent only if $\alpha = \beta \ge \gamma$. In this case, the criterion would be effective up to $r \le n^\alpha + \tfrac{1}{2}n^\gamma - 1$. If $n$ is sufficiently large, the optimal case is obtained when $\alpha = \beta = \lfloor (d-1)/2 \rfloor$ and $\gamma = d - 2\alpha$. This is precisely what one obtains by applying the proposed heuristic. Indeed, in the first step we would choose $\alpha \ge \beta = \lfloor (d-1)/2 \rfloor$. Then, $\alpha$ could either be $\lfloor (d-1)/2 \rfloor$ or $\lceil (d-1)/2 \rceil$ with the heuristic suggesting to pick $\alpha = \beta$. Finally, the value of $\gamma$ is set to $d - 2 \alpha$ so that $\gamma \le 2 \le \beta \le \alpha$. 
The remaining configurations do not result in a larger range of effective identifiability. 
If $n^\alpha \ge r \ge n^\beta \ge n^\gamma$, then the reshaped Kruskal criterion is effective for $r \le n^\beta + n^\gamma -2$. There is but one choice of $\beta$ that might result in a larger range than the proposed heuristic, namely $\beta = \lfloor (d-1)/2 \rfloor$, $\alpha = \lceil (d-1)/2 \rceil$ and $\gamma=1$, and this can only occur when $d$ is even. However, the resulting range is not optimal because $n \le \tfrac{1}{2}n^{d-2\lfloor(d-1)/2\rfloor} = \tfrac{1}{2} n^2$ (whenever $n \ge 2$) for even $d$, so that the proposed heuristic always covers a wider range. If $n^\alpha \ge n^\beta \ge r$, then the criterion is effective for $r \le n^\beta$, but it is immediately clear that this range is not optimal.
\end{proof}

\begin{proposition}
 Let $\Var{S} = \operatorname{Seg}(\Pj\F^{n_1} \times \cdots \times \Pj\F^{n_d})$ with $n_1 \ge \cdots \ge n_d \ge 2$ be an unbalanced Segre variety, i.e.,
 \(
  n_1 > 1 + \prod_{i=2}^d n_d - \sum_{i=2}^d (n_i - 1).
 \)
Then the reshaped Kruskal criterion in \refthm{thm_reshaped_kruskal} is effective for 
\[
 r \le \prod_{i=2}^{d-1} n_i + n_d - 2.
\]
Furthermore, this is the largest range in which \refthm{thm_reshaped_kruskal} applies.
\end{proposition}
\begin{proof}
For $d=3$, we may apply \refprop{prop_kruskal}. Since {the case} $r > n_1$ is not generically $r$-identifiable because of \cite[Theorem 3.1]{CGG2002} and \cite[Proposition 8.2]{BCO2013}, it follows that $r \le \tfrac{1}{2}(r + n_2 + n_3) - 1$ is the widest range in which Kruskal's criterion applies, concluding the proof in this case.

Let $d \ge 4$ in the remainder. Then, we observe that 
 \begin{align*}
  \prod_{i=2}^{d-1} n_i > n_1 > 1 + \prod_{i=2}^d n_i - \sum_{i=2}^{d} (n_i - 1)
 \end{align*}
 is inconsistent, as we should have that
\begin{align*}
1 
&> n_d \Bigl( 1 - \prod_{i=2}^{d-1} n_i^{-1} \Bigr) - \frac{ \sum_{i=2}^{d-1} (n_i - 1) }{ \prod_{i=2}^{d-1} n_i } + 2 \prod_{i=2}^{d-1} n_i^{-1} \\
&= n_d \Bigl( 1 - \prod_{i=2}^{d-1} n_i^{-1} \Bigr) - \frac{ \sum_{i=2}^{d-1} n_i }{ \prod_{i=2}^{d-1} n_i } + d \prod_{i=2}^{d-1} n_i^{-1}, \\
&\ge 2(1-2^{-d+2}) - (d-2)2^{-d+3} + d 2^{-d+2} \\
&= 2 - (d-1)2^{-d+3} + \tfrac{d}{2} 2^{-d+3} = 2 - (\tfrac{d}{2}-1) 2^{-d+3},
\end{align*}
where the second inequality is because of $n_i \ge 2$.
The right hand side is never less than $1$ if $d \ge 4$. Hence, $n_1 \ge \prod_{i=2}^{d-1} n_i$. It follows that the heuristic chooses $\vect{h} = \{1\}$, $\vect{k} = \{2,\ldots,d-1\}$, and $\vect{l} = \{d\}$. The situation $r \ge n_1$ is never generically identifiable in the unbalanced case because of \cite[Theorem 3.1]{CGG2002} and \cite[Proposition 8.2]{BCO2013}. Considering the case $n_1 \ge r \ge \Pi_\vect{k} \ge \Pi_\vect{l}$ leads precisely to the bound on $r$ as in the formulation of the proposition.

It follows from $n_1 \ge \prod_{i=2}^{d-1} n_i$ that $n_1$ is larger than every $\Pi_\vect{k}$ with $\{1\} \party \vect{k} \party \vect{l} = \{1,\ldots,d\}$ with both $\vect{k}$ and $\vect{l}$ nonempty. So, the conditions in \refthm{thm_reshaped_kruskal} can be satisfied only if $\vect{h} \subset \{1,\ldots,d\}$ contains at least ``$1$.'' Whatever the partition $\vect{h}\party\vect{k}\party\vect{l}=\{1,\ldots,d\}$ with $1\in\vect{h}$, {we have} $\delta < 0$ because otherwise the criterion is effective for $r$ larger than $n_1$, which is impossible. 
Therefore, the effective range of identifiability of \refthm{thm_reshaped_kruskal} is $r \le \Pi_\vect{k} + \Pi_\vect{l} - 2$ with $\Pi_\vect{k} \ge \Pi_\vect{l}$ and where $\vect{k} \party \vect{l} = \{1,\ldots,d\}\setminus\vect{h}$. It follows from $n_1 \ge \cdots \ge n_d \ge 2$ and the observation that $n_i a + \tfrac{1}{n_i}b > a + b$ when $a \ge b$ that the maximum is reached for $\vect{k} = \{2,\ldots,d-1\}$, concluding the proof.
\end{proof}

\subsection{Computational complexity}
In practice, we should also account for the substantial computational complexity of computing the Kruskal ranks. 
The following result should be well-known to the experts.

\begin{romanthm}\label{lem_complexity}
Let $\Var{X} \subset \Pj\F^{N}$. The computational complexity of checking that the Kruskal rank of $r$ points 
$p_1,p_2,\ldots,p_r \in \Var{X}$ is \emph{at least} $\kappa \le \min\{r, N\}$ by computing the ranks of $\binom{r}{\kappa}$ 
matrices of size $N \times \kappa$ is
\(
 \mathcal{O}\bigl( \binom{r}{\kappa} \kappa^2 N \bigr).
\)
It follows that the computational complexity of checking that the points $p_i$, $i=1,\ldots,r$, are in GLP using this method is
\[
\mathcal{O}\biggl( \binom{r}{N} N^3 \biggr) \text{ if } r > N, \text{ and } \mathcal{O}( r^2 N ) \text{ if } r \le N.
\]
\end{romanthm}

\begin{romanrem}
Verifying that the Kruskal rank is at least $2 \le \kappa \le r \le N$ is \emph{more} expensive than verifying that the same points are in GLP, because $\kappa^2 \binom{r}{\kappa} > r^2$ whenever $r \ge 3$.
\end{romanrem}

With the proposed heuristic the computational cost of verifying \refthm{thm_reshaped_kruskal}, in particular the cost 
of checking that the points on the third factor $\Var{S}_\vect{l}$, i.e., $\pi_\vect{l}(p_1),\ldots,\pi_\vect{l}(p_r)$, are 
in GLP, may be prohibitive. The reason is that the cost is at least $\binom{r}{n_{l_1}}n_{l_1}^3$, which is almost 
invariably too expensive if $r$ is large relative to $n_{l_1}$. For instance, if $n_1=100$, $n_2=90$, and $n_3=10$ 
with $r = 90$, then checking GLP on the third factor requires $1000\binom{90}{10}$ operations, which would take 
roughly $6$ days on a computer that completes $10$Gflop/s.
Therefore, we recommend verifying only that the Kruskal rank of aforementioned projected points on the third 
factor $\Var{S}_\vect{l}$ is greater than $1$ by testing for all $\binom{r}{2}$ pairs of points that the points are 
distinct in projective space. This can be accomplished with $\mathcal{O}( \binom{r}{2} n_{l_1} )$ operations, 
which increases only polynomially in $r$. In the previous example, this would reduce the computational cost to only
 $10\binom{90}{2}$ operations, which can be completed in only $4$ microseconds on the same hypothetical computer as before. 

In summary, the following corollary\footnote{The three-factor version of this criterion is sometimes attributed to Harshman \cite{Harshman1972}, however his proof only covers the case $n_1 \ge n_2 \ge n_3 = 2$.} is usually more appealing because of its reduced computational complexity. Its effectiveness is a consequence of \refthm{thm_reshaped_kruskal} and \reflem{lem_complexity}.

\begin{romancor}
Let 
\(
\Var{S} = \operatorname{Seg}( \Pj\F^{n_1} \times \Pj\F^{n_2} \times \cdots \times \Pj\F^{n_d} ).
\)
Let $\Pi_\vect{m} = \prod_{\ell \in \vect{m}} n_\ell$ for any $\vect{m} \subset \{1,2,\ldots,d\}$.
Let $\vect{h} \party \vect{k} \party \vect{l} = \{1,2,\ldots,d\}$, such that $\Pi_{\vect{h}} \ge \Pi_{\vect{k}} \ge \Pi_\vect{l}$. 
Let $p \in \langle p_1, p_2, \ldots, p_r\rangle$ with $p_i = \sten{a}{i}{1}\otimes\cdots\otimes\sten{a}{i}{d} \in \widehat{\Var{S}}_0$.
If both matrices
$A_{\vect{h}}$ and $A_{\vect{k}}$ are of rank $r$ and the Kruskal rank of $A_{\vect{l}}$ is at least $2$,
then $p=\alpha_1 p_1 + \cdots + \alpha_r p_r$ is $r$-identifiable for every $\alpha_i \in \F_0$. 
This criterion is effective in the entire range, i.e., for all $r \le \Pi_\vect{k}$. 

The computational complexity of verifying this criterion is
\[
\mathcal{O}\biggl( r^2 (\Pi_\vect{h} + \Pi_\vect{k}) + \binom{r}{2} \Pi_\vect{l} \biggr);
\]
for fixed $d$, it thus has polynomial complexity in the size of the input $r(n_1 + n_2 + \cdots + n_d)$.
\end{romancor}

Employing the heuristic from \refsec{sec_heurisic_reshaping} is advised for obtaining a large range of effectiveness with the above criterion.

\section{Symmetric identifiability}\label{sec_symmetric}
The main purpose of this section is introducing a technique for 
investigating specific identifiability in the symmetric setting based on the Hilbert function. 

\subsection{Basic results}

A well-known result on effective symmetric identifiability is the catalecticant method of \cite[5.4]{IK1999}. It is stated below only for the even degree case as the reshaped Kruskal criterion applies in a wider range for odd degree.
\begin{proposition}[Iarrobino and Kanev \cite{IK1999}]
Let $d = 2m$, and let $\Var{V} = \Pj v_{d}(\F^{n+1})$ be the Veronese variety. Let $p = p_1 + \cdots + p_r$ with $p_i = \sten{a}{i}{\circ d} \in \widehat{\Var{V}}$ be a given decomposition. Let the most square symmetric flattening of $p$ be denoted by
\[
 C = \sum_{i=1}^r (\sten{a}{i}{\circ m}) (\sten{a}{i}{\circ m})^T.
\]
If $\rank{C} = r$ and $r \le \binom{n+m}{m} - (n+1)$, then the kernel of $C$ is the ideal $I_{Z,m}$ of polynomials of degree $m$ simultaneously vanishing on $Z=\{\vect{a}_1, \ldots, \vect{a}_r\}$. If additionally the degree of the closure of the zero set of $I_{Z,m}$ is $r$, then $p$ is $r$-identifiable. This criterion is effective for all 
\[
 r \le \binom{n+m}{m} - (n+1).
\]
\end{proposition}
\begin{proof}
 Effectiveness was proved in \cite[Theorem 2.4]{OO2013}. 
\end{proof}
An implementation of the catalecticant method---which is easily adapted to a criterion for effective specific identifiability as outlined above---is also described in \cite{OO2013}.

The reshaped Kruskal criterion for general tensors is also effective when applied to 
reshaped symmetric tensors. If $d_1 + d_2 + d_3 = d$ is a partition of $d$, then reshaping a rank-$1$ 
symmetric tensor can be thought of as
\begin{align*}
 \Pj v_d( \F^{n+1} ) &\to \operatorname{Seg}\bigl(\Pj v_{d_1}(\F^{n+1}) \times \Pj v_{d_2}(\F^{n+1}) 
\times \Pj v_{d_3}(\F^{n+1}) \bigr) \\
 [\sten{a}{i}{\otimes d}] &\mapsto
 [\sten{a}{i}{\otimes d_1} \otimes \sten{a}{i}{\otimes d_2} \otimes \sten{a}{i}{\otimes d_3}]
\end{align*}
The map can be extended linearly to define reshaping for an arbitrary $d$-form. 
The image of this map is contained in the projectivization of
\(
S^{d_1} \F^{n+1} \otimes S^{d_2}\F^{n+1} \otimes S^{d_3}\F^{n+1} \cong \F^{\binom{n+d_1}{d_1}} \otimes
 \F^{\binom{n+d_2}{d_2}} \otimes \F^{\binom{n+d_3}{d_3}}.
\)

\begin{lemma}\label{lab_symm_glp}
 Let $\Var{S} = \operatorname{Seg}(\Pj\F^{n+1} \times \cdots \times \Pj\F^{n+1})$ be a $d$-factor Segre variety. 
Let $\Var{V} = \Pj S^d\F^{n+1} \cap \Var{S}$ be the variety of symmetric rank-$1$ tensors in $\Pj(\F^{n+1} 
\otimes \cdots \otimes \F^{n+1})$. Then, there exists a dense Zariski-open subset $G \subset \Var{V}^{\times r}$ 
with the property that for every $\vect{h} \subset \{1,2,\ldots,d\}$ and every $(p_1,p_2,\ldots,p_r)\in G$, the points 
$(\pi_\vect{h}(p_1), \pi_\vect{h}(p_2), \ldots, \pi_\vect{h}(p_r)) \in \Var{S}_\vect{h} \cap \Pj 
S^{|\vect{h}|} \F^{n+1}$ are in GLP.
\end{lemma}
\begin{proof}
The proof follows along the same lines as the proof of \reflem{lem_max_glp}.
\end{proof}

The foregoing lemma in combination with \refprop{prop_kruskal} yields a symmetric
 version of the reshaped Kruskal condition in \refthm{thm_reshaped_kruskal}.
\begin{corollary}\label{cor_symm_reshaped_Kruskal}
Let 
\(
\Var{S} = \operatorname{Seg}(\Pj\F^{n+1} \times \cdots \times \Pj\F^{n+1})
\)
and
\(
\Var{V} = 
\Pj S^d \F^{n+1} \cap \Var{S}.
 \)
Let $p \in \langle p_1, \ldots, p_r \rangle$ with $p_i = \sten{a}{i}{\otimes d} \in \widehat{\Var{V}}$. Let $\Gamma_k =
 \binom{k+n}{n}$ for any $k = 1,2,\ldots,d$. Let $d_1+d_2+d_3=d$ be a partition of $d$, such that 
$d_1 \ge d_2 \ge d_3$. 
Let $\kappa_1$, $\kappa_2$, and $\kappa_3$ denote the Kruskal ranks of 
 \(\{ \sten{a}{1}{\otimes d_1}, \ldots, \sten{a}{r}{\otimes d_1} \}\), 
 \(\{ \sten{a}{1}{\otimes d_2}, \ldots, \sten{a}{r}{\otimes d_2} \}\), and
 \(\{ \sten{a}{1}{\otimes d_3}, \ldots, \sten{a}{r}{\otimes d_3} \}\)
 respectively.  Then, $p$ is $r$-identifiable if 
 \(
  r \le \frac{1}{2} (\kappa_1 + \kappa_2 + \kappa_3) - 1.
 \)
 Furthermore, letting $\delta = \Gamma_{d_2} + \Gamma_{d_3} - \Gamma_{d_1} - 2$, this criterion is effective if
 \begin{align*}
  r \le \Gamma_{d_1} + \min\{ \tfrac{1}{2} \delta, \delta \}.
 \end{align*}
For large $n$, the maximum range of effective $r$-identifiability is attained for $d_1 = d_2 = 
\lfloor \tfrac{1}{2} (d-1) \rfloor$ and $d_3 = d - 2d_1$:
\[
 r \le 
 \begin{cases}
 \tfrac{3}{2}n + \tfrac{1}{2} &\text{if } d = 3,\\
 2n &\text{if } d = 4,\\
 \binom{d_1 + n}{d_1} + \frac{1}{2}\binom{d_3+n}{d_3} - 1 &\text{if } d \ge 5.
 \end{cases}
\]
\end{corollary}
\begin{proof}
 The upper bound on the range of effective identifiability can be proved in exactly the same way as \refprop{prop_best_range}.
\end{proof}

This criterion is effective in the entire range where generic $r$-identifiability holds for a small number of spaces. We call $S^d \F^{n+1}$ effectively identifiable if there exist effective criteria for specific $r$-identifiability for every $\Pj v_d(\F^{n+1})$ that is generically $r$-identifiable.

\begin{proof}[Proof of \refthm{thm_effective_crit}, part I]
$S^3 \F^{3}$ is the only ``normal'' case in the theorem. It is generically $r$-identifiable for $r \le 3$, the generic rank is $4$ and the space is not perfect (or equiabundant). \refcor{cor_symm_reshaped_Kruskal} applies up to $3$, hence concluding this case.

$S^3 \F^4$ is a perfect space and one of the exceptionally identifiable cases in \refthm{thm_perfect}. Generic $r$-identifiability holds up to $r=5$ and \refcor{cor_symm_reshaped_Kruskal} establishes effective specific identifiability up to $r=5$ as well.

$S^3 \F^5$ is a perfect space with generic rank $7$ that is not generically $7$-identifiable because of \refthm{thm_perfect}. It is generically $r$-identifiable for $r\le6$ and \refcor{cor_symm_reshaped_Kruskal} is an effective criterion in this range.

Both $S^3 \F^6$ and $S^6 \F^3$ are effectively identifiable because \refcor{cor_symm_reshaped_Kruskal} applies up to $r=8$, both $\Pj v_3(\F^6)$ and $\Pj v_6(\F^3)$ are generically $r$-identifiable for $r \le 8$, and they are $9$-tangentially weakly defective by \refthm{thm_cov2016}.

$S^4 \F^3$ is a perfect space with generic rank $5$. \refcor{cor_symm_reshaped_Kruskal} yields an effective specific identifiability criterion up to $r = 4$. Since $S^4 \F^3$ is not $5$-identifiable by \refthm{thm_perfect}, the proof of this case is concluded.

Binary forms of even degree $\Pj v_{2k}(\F^2)$ are generically $k$-identifiable, but not identifiable for the generic rank $k+1$. For $k=2$, \refcor{cor_symm_reshaped_Kruskal} yields effective specific identifiability up to $2$. For $k\ge3$, taking $d_1 = d_2 = k-1$ and $d_3 = 2$, \refcor{cor_symm_reshaped_Kruskal} implies effective specific identifiability up to $r \le (1+k-1) + \tfrac{1}{2}(1+2) - 1 = k + \tfrac{1}{2}$. For $\Pj v_{2k+1}(\F^2)$ generic $r$-identifiability exceptionally holds up to $r = k+1$ by \refthm{thm_perfect}. \refcor{cor_symm_reshaped_Kruskal} then implies effective specific $r$-identifiability up to $r \le (k+1)+\tfrac{1}{2} 2 - 1 = k+1$ by choosing $d_1 = d_2 = k$ and $d_3=1$ if $d \ge 5$, and the case $d=3$ yields $r\le 2$. This proves effective identifiability of $S^d \F^2$ for all $d\ge3$.
\end{proof}

An interesting case for which we lack an effective criterion of specific identifiability is $S^5 \F^3$ which is a perfect space that is exceptionally generically $7$-identifiable by \refthm{thm_perfect}, while \refcor{cor_symm_reshaped_Kruskal} only applies up to $r \le 6$.

\begin{romanex}
Let $d = 6$ and $n=3$. According to the corollary, applying the reshaped Kruskal criterion to a generic symmetric tensor of rank $r \le \tfrac{3}{2}\binom{3+2}{2}-1 = 14$ will certify uniqueness with probability $1$. Let us generate a random real symmetric tensor in $S^6\R^{4}$ by executing the following Macaulay2 code:
\begin{verbatim}
n = 3; r = 14;
A_0 = matrix apply(n+1, j->apply(r, k->random(-99,99)));
\end{verbatim}
This matrix \texttt{A\_0} naturally corresponds with the symmetric tensor $\tensor{A} = \sum_{i=1}^{14} \sten{a}{i}{\otimes 6}$, where $\sten{a}{i}{}$ is the $i$th column of \texttt{A\_0}. Its $r$-identifiability can be verified with the reshaped Kruskal criterion by applying Kruskal's criterion to $\sum_{i=1}^{14} \sten{a}{i}{\otimes2} \otimes \sten{a}{i}{\otimes 2} \otimes \sten{a}{i}{\otimes2}$. To this end, we should simply compute the Kruskal rank of $A \odot A \in \R^{4^2 \times 14}$. Note that the columns of this matrix live in $S^2 \R^4$, i.e., they can be considered as vectorizations of symmetric $4 \times 4$ matrices. The Kruskal rank can be computed with the functions in \texttt{reshapedKruskal.m2} as follows:
\begin{verbatim}
kruskalRank(kr(A,{0,0}))
o1 = 10
\end{verbatim}
Note that this is the maximum values because $\dim S^2 \R^4 = 10$.
Since $r = 14 \le \tfrac{3}{2}10 - 1 = 14$, Kruskal's criterion holds, and hence the chosen tensor is $14$-identifiable.
\end{romanex}

\subsection{The Hilbert function}
In this section we introduce some algebraic methods for the detection of the identifiability
of symmetric tensors, namely the Hilbert function of a set of points in a
projective space and their $h$-vector.
Both of these methods are widely used in algebraic geometry, and their
application to the identifiability problem has been considered before in the
literature; see, e.g., \cite{BB2012, BC2012, BC2013, BGL2013}. Yet, we
believe that the interactions between the Hilbert function and tensor
analysis have not yet been fully explored (see also \cite{CM2015}). {We
will employ these techniques in the next section for proving the last remaining 
case of \refthm{thm_effective_crit}.}

Consider a polynomial ring $R=\C[x_0,\dots,x_n]$ and 
the linear space $R_d$ of forms of degree $d$. Let $Z$ be a finite set in $\Pj\C^{n+1}$.
Call $I_Z$ the homogeneous ideal of the set $Z$. Then there is an
exact sequence of \emph{graded} modules:
\(
0\to I_Z \to R\to R/I_Z\to 0.
\)

\begin{definition} The \emph{Hilbert function} $H_Z$ of the set $Z$ associates to each integer
$d$ the dimension $H_Z(d)$ of the linear space $(R/I_Z)_d$.
\end{definition} 

\begin{remark}\label{coho}
There is an interpretation of the Hilbert function in terms of the residue of forms at points.
For a form $f\in R_d$ and a point $P\in Z$, the evaluation $f(P)$ is not well defined,
as it depends on the choice of coordinates for $P$, which is fixed only
up to scalar multiplication.
However, if we consider the residues of \emph{all} forms in a linear space
at \emph{all} possible homogeneous coordinates of the points of $Z$, then
we get a well defined subspace of $\C^\ell$, where $\ell$ is the cardinality of $Z$.
In this sense, if we take the residue of all forms of degree $d$, the dimension
of the subspace of $\C^\ell$ that we obtain is equal to $H_Z(d)$.

A precise algebraic formulation of this principle is easy in the theory of sheaves.
Call $\oo$ the structure sheaf of $\Pj\C^{n+1}$ and $\oo_Z$ the structure sheaf
of $Z$, which is a skyscraper sheaf supported at the $\ell$ points of $Z$.
Then for any degree $d$ we have a well-defined surjective map of sheaves $\oo(d)\to \oo_Z$
whose kernel is the ideal sheaf $\ii_Z(d)$ of $Z$.
Taking global sections, we get an exact sequence of vector spaces
$0\to H^0(\ii)\to H^0(\oo(d))\to H^0(\oo_Z).$
Since $\oo_Z$ is a skyscraper sheaf, then $H^0(\oo_Z)$ can be non-canonically
identified with $\C^\ell$, while $H^0(\oo(d))$ is $R_d$. The 
 left-hand map $\rho_d: H^0(\oo(d))\to H^0(\oo_Z)$ corresponds to taking
residues, as specified above. Thus the rank of $\rho_d$ is
the value of the Hilbert function $H_Z(d)$.
 \end{remark}

Some well-known properties of the Hilbert function are recalled next.

\begin{proposition}\label{easy}\par\noindent
\begin{enumerate}
\item[(i)] $0=H_Z(-1)=H_Z(-2)=\dots$;
\item[(ii)] $H_Z(0)=1$;
\item[(iii)] $H_Z(1)<n+1$ exactly when $Z$ is contained in a hyperplane;
\item[(iv)] $H_Z(d)<\binom{n+d}d$ if and only if $Z$ is contained in a
hypersurface of degree $d$;
\item[(v)] $H_Z(d)\leq H_Z(d+1)$;
\item[(vi)] $H_Z(d)$ cannot be bigger than the cardinality $\ell$ of $Z$;
\item[(vii)] for all $d\gg 0$ then $H_Z(d)=\ell$, the cardinality of $Z$; and
\item[(viii)] if $Z'\subset Z$ then $H_Z(d)\geq H_{Z'}(d)$ for all $d$.
\end{enumerate}
\end{proposition}

From now on, we write $\ell_Z$ for the cardinality of a finite set $Z$.
A bit more difficult, but still straightforward, is the proof of the next property.

\begin{proposition}\label{no0}
If $H_Z(d_0)=H_Z(d_0+1)$ for some $d_0\geq 0$, then $H_Z(d)=\ell_Z$ for all $d\geq d_0$.
\end{proposition}

The difference
$h_Z(d)= H_Z(d)-H_Z(d-1)$
is always non-negative by \refprop{easy}(v). Furthermore, by (i) and (ii) of \refprop{easy} we get $h_Z(0)=1$,
and from Proposition \ref{no0} it follows that if $h_Z(d)=0$ for some $d>0$, then $h_Z(d')=0$ for 
all $d'\ge d$.

\begin{definition} 
Let $Z$ be a finite set. The \emph{$h$-vector} of $Z$ is the sequence of integers
$(h_Z(0),h_Z(1),\dots ,h_Z(c))$
where $c$ is the maximum such that $H_Z(c-1)<\ell_Z$, i.e., the maximum such that $h_Z(c)>0$.
\end{definition}

The basic properties of the $h$-vector can be summarized as follows.

\begin{proposition}\label{easyh}\par\noindent
\begin{itemize}
\item[(i)] $h_Z(0)=1$;
\item[(ii)] $h_Z(i)>0$ for all $i$;
\item[(iii)] $h_Z(1)$ is the dimension of the projective linear span of $Z$;  
\item[(iv)] If $(h_Z(0),\dots,h_Z(c))$ is the $h$-vector of $Z$, then $H_Z(c)=\ell_Z$
and $H_Z(i)<\ell_Z$ for $i=0,\dots, c-1$; and
\item[(v)] $\sum_{i=0}^c h_Z(i) = H_Z(c)=\ell_Z.$
\end{itemize}
\end{proposition}

\begin{proposition}\label{subb}
If $Z'\subset Z$ then $h_{Z'}(d)\leq h_Z(d)$ for all $d$.
\end{proposition}
\begin{proof}
The $h$-vector $h_Z$ of $Z$ corresponds to the Hilbert function of an Artinian reduction
 $R/(\ii_Z+L)$ with $L$ a generic linear form (see e.g. \cite[Remark 6.2.8]{Peeva}), and 
an Artinian reduction of $Z'$ is a quotient of $R/(\ii_Z+L)$.
\end{proof}

\begin{remark} \label{separ}
Assume that $H_Z(d)=\ell_Z$. Then the map $\rho_d: H^0(\oo(d))\to H^0(\oo_Z)$
introduced in Remark \ref{coho} surjects. {Thus all the elements
of $H^0(\oo_Z)\simeq\C^{\ell_Z}$ sit in the image of the evaluation map.} In particular,
the vector $\left[\begin{smallmatrix} 1 & 0 & \cdots & 0\end{smallmatrix}\right]$
is in the image. This implies that there is a form $f$
of degree $d$ vanishing at all the points of $Z$ except for the first one.
Geometrically this means that there exists a hypersurface of degree $d$ in $\Pj\C^{n+1}$ 
that contains all but one points of $Z$.
As the same phenomenon occurs for all elements of the natural basis of
$H^0(\oo_Z)=\C^{\ell_Z}$, we can find for every $P\in Z$ a hypersurface of degree $d$
that contains $Z\setminus \{P\}$ and excludes $P$.
Thus, if $H_Z(d)=\ell_Z$, then we will say that \emph{hypersurfaces of degree $d$
separate the points of $Z$}.
\end{remark}

The Hilbert function is closely tied with the linear properties of the images of $Z$ under Veronese maps of increasing degrees.

\begin{proposition}\label{HilbVer}
$H_Z(d)$ is equal to the (projective) dimension of the linear span of the image 
of $Z$ in $v_d$ plus $1$:
\(H_Z(d)=\dim \langle v_d(Z) \rangle+1.\)
Consequently, $H_Z(d)=\ell_Z$ if and only if the points of 
$v_d(Z)$ are linearly independent.
\end{proposition}
\begin{proof} The projective dimension $\delta$ of the linear span $\langle v_d(Z) \rangle$ 
is equal to $N$ minus the affine dimension of the space of linear forms
whose corresponding hyperplanes in $\Pj\C^{N+1}$ contain $v_d(Z)$. Thus $\delta$ is
equal to $N-\dim(J_1)$, where $J$ is the homogeneous ideal of $v_d(Z)$ in $\Pj\C^{N+1}$.
Now notice that $N+1=\binom{n+d}d=\dim R_d$. Moreover,
$J_1$ corresponds to the space of forms in $R_d$ which contain $Z$.
Since, by definition, $H_d(Z)=\dim R_d-\dim I_d$, where $I\subset R$ is
the homogeneous ideal of $Z$ in $\Pj\C^{n+1}$, the claim follows.
\end{proof} 

If $Z$ is the union of two disjoint sets $A$ and $B$, then the Hilbert function provides
a way to compute the dimension of the intersection 
$\langle v_d(A) \rangle \cap \langle v_d(B) \rangle$.

\begin{proposition}\label{Grass}If $A$ and $B$ are subsets of $\Pj\C^{n+1}$
and both $v_d(A)$ and $v_d(B)$ are linearly independent sets, then
\(
\dim(\langle v_d(A) \rangle \cap \langle v_d(B) \rangle)= \ell_A+\ell_B-H_Z(d)-1,
\)
where $Z=A\cup B$.
\end{proposition}
\begin{proof}
We use the Grassmann formula:
\begin{align*}
\dim\bigl(\langle v_d(A)\rangle \cap \langle v_d(B)\rangle\bigr)
= \dim(\langle v_d(A)\rangle) +\dim(\langle v_d(B)\rangle) - \dim \left(\langle v_d(A)\rangle+\langle v_d(B)\rangle\right).
\end{align*}
Since $v_d(A)$ and $v_d(B)$ are linearly independent, it follows that $\dim(\langle v_d(A)\rangle)=\ell_A-1$ and 
$\dim(\langle v_d(B)\rangle)=\ell_B-1$. Moreover by Proposition \ref{HilbVer}, 
\(
\dim (\langle v_d(A)\rangle+\langle v_d(B)\rangle) = \dim (\langle v_d(A)\cup v_d(B)\rangle) = H_Z(d)-1.
\)
The claim follows.
\end{proof}

We introduce a fundamental property of finite sets of points in a projective space. 

\begin{definition} We say that a finite set of points $Z\subset \Pj\C^{n+1}$ 
satisfies the \emph{Cayley-Bacharach property in degree $d$}---abbreviated as $\mathit{CB}(d)$---if 
\emph{for every} $P\in Z$ every form of degree $d$ vanishing at $Z\setminus\{ P\}$
also vanishes at $P$.
\end{definition}

If $Z$ satisfies $\mathit{CB}(d)$, then hypersurfaces of degree $d$
cannot separate the points of $Z$; in some sense $\mathit{CB}(P)$ is the exact opposite of separation.
Thus, if $Z$ satisfies $\mathit{CB}(d)$, then $H_Z(d)<\ell_Z$ and $h_Z(d+1)>0$.
However, the converse is false. For instance, the set $Z$ consisting of four points in $\Pj\C^3$, 
three of them aligned,
does not satisfy $\mathit{CB}(1)$, while $H_Z(1)<4$.

The main reason for introducing the $\mathit{CB}(d)$ property lies in the following result, which strongly
bounds the Hilbert functions of set with a Cayley-Bacharach property.

\begin{theorem}[Geramita, Kreuzer, and Robbiano \cite{GKR1993}]\label{CBprop}
The $h$-vector of a set of points $Z$ which satisfies $\mathit{CB}(d)$ has the following property:
for all $k\geq 0$,
$$ h_Z(0)+h_Z(1)+\cdots + h_Z(k)\leq h_Z(d+1-k)+\cdots +h_Z(d)+h_Z(d+1).$$
\end{theorem}

We proceed by showing the link between Hilbert functions of finite sets and the identifiability
problem for symmetric tensors.
Let $\tensor{A}\in S^d(\C^{n+1})$ be a symmetric tensor with two different
\emph{``minimal''}  decompositions
$
\tensor{A} = \vect{v}_1^{\circ d}+\cdots +\vect{v}_r^{\circ d} = 
\vect{w}_1^{\circ d}+\cdots \vect{w}_s^{\circ d}.
$
In the present context, minimality of the decompositions means that $\tensor{A}$ does not 
lie in the span of a proper subset of the $\vect{v}_i^{\circ d}$'s or of the $\vect{w}_j^{\circ d}$'s.
Let $P_i=[\vect{v}_i]$ and $Q_j=[\vect{w}_j]$ be the points of $\Pj\C^{n+1}$
corresponding to the elements of the decompositions. Define
$$
A=\{P_1,\dots,P_r\},\quad  
B=\{Q_1,\dots,Q_s\}, \text{ and}\quad 
Z=A\cup B.
$$
Then, the projective point $[\tensor{A}]\in \Pj(S^d \C^{n+1})$ belongs
to both spans $\langle v_d(A) \rangle$ and $\langle v_d(B) \rangle$. 
The minimality assumption means that $[\tensor{A}]$ does not
belong to the linear span of any proper subset of either $v_d(A)$ or $v_d(B)$.
So, the intersection $\langle v_d(A) \rangle \cap \langle v_d(B\setminus A) \rangle$ is necessarily non-empty
and $[\tensor{A}]$ belongs to the span of $\langle v_d(A) \rangle \cap \langle v_d(B\setminus A) \rangle$
and $v_d(A)\cap v_d(B)$.
In particular, it follows that the points of $\langle v_d(Z) \rangle$
are not linearly independent. Hence $H_Z(d) < \ell(Z)$, so that $h_Z(d+1)>0$
by Proposition \ref{easyh}(iv). 

In applications, we are mainly confronted with sets $A$ and $B$ that are in GLP, essentially because of 
\reflem{lem_max_glp}. 
In terms of the Hilbert function, $Z$ is in GLP if
and only if for every subset $Z'$ of $Z$ of $s\leq n+1$ points we have that
$H_{Z'}(1)=s$ and  $h_{Z'}(1)=s-1.$
In other words, if we consider an $(n+1) \times \ell_Z$ matrix $M$ whose columns consist of the
projective coordinates for the points of $Z$, then $Z$ is in
GLP if and only if every set of $\min\{\ell_Z,n+1\}$ columns of $M$ is linearly independent.

\section{An effective criterion for $S^4\C^4$}\label{sec_first_effective}
We show how an analysis of the Hilbert function yields an effective criterion for
symmetric tensors of type $4\times 4\times 4\times 4$.
The goal consists of affirming the $r$-identifiability of a tensor
\begin{align}\label{eqn_rank_decomposition_d4}
 \tensor{A} = \vect{v}_1^{\circ 4} + \cdots + \vect{v}_r^{\circ 4}
\end{align}
for any value of $r$.
The results of \cite{C2001, Mella2006} entail that generic tensors of rank $r=8$
in $\Pj(S^4 \C^4)$ are (exceptionally) not $8$-identifiable; they admit two distinct complex decompositions; see, e.g., 
\cite[Section 2]{COV2016}.
Consequently, decompositions with $r \ge 9$ are also not generically $r$-identifiable. On the other hand, it was proved in \cite{BL} that generic tensors of rank $r \le 7$ in $\Pj(S^4 \C^4)$ are identifiable.\footnote{Combining 
the Alexander--Hirschowitz theorem \cite{AH1995} with \cite[Corollary 4.5]{Mella2006} also yields this result.}
An effective criterion for $4 \times 4 \times 4 \times 4$ symmetric tensors should thus certify
generic $r$-identifiability for all $r \le 7$. The reshaped Kruskal criterion (\refcor{cor_symm_reshaped_Kruskal}) is effective in the symmetric setting if
\(
 r \le \binom{3+2}{2} + \min\{ \tfrac{1}{2} \delta, \delta \} = 10-4 = 6
\)
because $\delta = 4+4-\binom{3+2}{2}-2 = -4$. 
As far as we are aware, no effective criterion is known for $r=7$ in the literature. 
In this section, we derive an effective criterion for specific $7$-identifiability of 
tensors in $\Pj(S^4 \C^4)$, hereby concluding the proof of \refthm{thm_effective_crit}.

\subsection{Theory}
Assume that we are given the decomposition \refeqn{eqn_rank_decomposition_d4} with length $r=7$
and that we should determine if $\tensor{A}$ is $7$-identifiable. We start by making two 
assumptions. First, we can assume that the given decomposition is minimal. It is trivial to 
ascertain minimality by checking that $H_A(4) = 7$, which is an easy rank computation. 
If the decomposition is not minimal, then it is not of rank $7$, and so not $7$-identifiable. 
The second assumption is
that the set $A=\{[\vect{v}_1], \dots ,[\vect{v}_7]\}$ is in GLP, a condition which is also easy 
to verify. By \reflem{lem_max_glp}, the subset of points not in GLP on $\Sec{7}{v_4(\C^4)}$ forms
a Zariski-closed set. Hence, this assumption will not alter the effectiveness of our criterion.

We need to show that a different decomposition
$
\tensor{A}= \vect{w}_1^{\circ 4} + \cdots + \vect{w}_s^{\circ 4}
$
with $s\leq 7$ does not exist. Arguing by contradiction, we may assume that a second decomposition
exists and investigate which consequences it has on the geometry of the set $A$. We can assume 
without loss of generality that this alternative decomposition is minimal.
In the remainder, let $B=\{[\vect{w}_1], \dots ,[\vect{w}_s]\}$ and $Z=A\cup B$. 

\begin{proposition} \label{prop_disjointAB}
If alternative decompositions exist, then we can choose an alternative decomposition with $A$ and $B$ disjoint.
\end{proposition}

The proof of this result is delayed until after \refprop{prop_twisted_cubic}.

\begin{proposition} \label{prop_CB}
Alternative decompositions exist only if $Z$ satisfies $\mathit{CB}(4)$.
\end{proposition}
\begin{proof} Assume it does not. Then, there exist a $P\in Z$ and a form of degree $4$
that contains $Z'=Z\setminus\{P\}$ but excludes $P$. Thus, the homogeneous
ideals satisfy $\dim (I_Z)_4 < \dim (I_{Z'})_4$, so that $H_Z(4)>H_{Z'}(4)$.
It follows that $h_Z(q)>h_{Z'}(q)$ for some value $q\leq 4$. Since $h_Z(i)\geq h_{Z'}(i)$
for all $i$ by \refprop{subb}, and $\sum h_Z(i)=\ell_Z=1+\ell_{Z'}=1+\sum h_{Z'}(i)$,
it follows that  $h_Z(q)=1+h_{Z'}(q)$ and $h_Z(i)=h_{Z'}(i)$ for $i\neq q$. Thus,
$ H_Z(4)= H_{Z'}(4)+1.$

Now assume that $P\in A$ and recall that we may  assume $A\cap B=\emptyset$ by Proposition
\ref{prop_disjointAB}.
Setting $A'=A\setminus\{P\}$, we get from \refprop{Grass} that
\begin{align*}
\dim(\langle{}v_4(A)\rangle{}\cap \langle{}v_4(B)\rangle{})= \ell_A+\ell_B-H_Z(4)-1 
&= \ell_{A'}+\ell_B-H_{Z'}(4)-1 \\
&=\dim(\langle{}v_4(A')\rangle{}\cap \langle{}v_4(B)\rangle{}),
\end{align*}
so that $\langle{}v_4(A)\rangle{}\cap \langle{}v_4(B)\rangle{}=\langle{}v_4(A')\rangle{}\cap \langle{}v_4(B)\rangle{}$. 
Consequently, $\tensor{A}$ belongs to $v_4(A')$, contradicting the assumption of minimality.
If $P\in B$ we similarly obtain that $\tensor{A}$ belongs to the span of
$v_4(B\setminus\{P\})$, contradicting the minimality of $B$.
\end{proof}

\begin{proposition}\label{prop_twisted_cubic}
Alternative decompositions exist only if $s=|B|=7$. The $h$-vector of $Z$ is
$(1,3,3,3,3,1)$ and $Z$ is contained in an irreducible twisted cubic curve.
\end{proposition}
\begin{proof} 
Since $A$ is in GLP, the $h$-vector of $A$ is $(1,3,3)$.
Indeed, $h_A(0)=1$ is obvious, while $h_A(1)=3$ because $A$ spans $\Pj\C^4$.
So by \refprop{easyh} it just remains to prove that $H_A(2)=7$. For any $P\in A$, 
divide the remaining $6$ points in two set of three points each, and then take the two
planes spanned by the two sets. As $A$ is in GLP, no four points of $A$ 
belong to a plane, so that the two planes define a quadric that contains $A\setminus\{P\}$
and misses $P$. Thus, $A$ is separated by quadrics and $H_A(2)=7$ by Remark \ref{separ}.

$Z$ satisfies $\mathit{CB}(4)$ by \refprop{prop_CB}, and hence, by \refthm{CBprop},
\begin{align*}
h_Z(5) &\ge h_Z(0) = 1, \\
h_Z(4)+h_Z(5) &\ge h_Z(0)+h_Z(1) = 4, \text{ and }\\
h_Z(3)+h_Z(4)+h_Z(5) &\ge h_Z(0)+h_Z(1)+h_Z(2) = 4+h_Z(2).
\end{align*}
Since $h_Z(2)\geq h_A(2)=3$ by \refprop{subb},
$\ell_Z\geq \sum_{i=0}^5 h_Z(i) \geq 14$ so that $s\geq 7$. It follows that $s=7$ and 
$\ell_Z=\sum_{i=0}^5 h_Z(i)=H_Z(5)=14$, and, hence, $h_Z(2)=3$.
In particular $H_Z(2)=7$, so $Z$ is contained in three linearly independent {quadric} surfaces.
Clearly these quadric surfaces cannot meet in a finite number of points, since
$\ell_Z>8$. We will prove that $C$ is a twisted cubic curve that contains $Z$.

Notice that $h_Z(3)$ cannot be bigger that $3$, because $h_Z(4)+h_Z(5)\geq 4$.
If $h_Z(3)\leq 2$, then by \cite[Theorem 3.6]{BGM1994} and its proof
one has also $h_Z(4),h_Z(5)\leq 2$, contradicting
$h_Z(3)+h_Z(4)+h_Z(5)\geq h_Z(2)+h_Z(1)+h_Z(0)=7.$ Hence $h_Z(3)=3$.
It also follows that $h_Z(4)+h_Z(5)\leq 4$. Thus, equality holds. 
If $h_Z(4)\leq 1$ then also $h_Z(5)\leq 1$ by \cite[Theorem 3.6]{BGM1994} again,
which is a contradiction. 
Hence, there are only two possibility left for the $h$-vector of $Z$, namely
\(h_Z=(1,3,3,3,2,2)\) or \(h_Z=(1,3,3,3,3,1)\).
Next, we use again \cite[Theorem 3.6]{BGM1994}. In the former case, since $h_Z(4)=h_Z(5)=2$, then
there exists a curve $C$ of degree $2$ containing a subset $Z'\subset Z$,
and the ideal of $C$ coincides with the ideal  of $Z'$ up to degree $5$.
If $C$ is a conic, then it must contain at least $11$ points of $Z'$, hence
at least $4$ points of $A$, which is impossible since a conic
is a plane curve and $A$ is in GLP. If $C$ is a disjoint union of lines
then it must contain at least $12$ points of $Z$, hence at least $5$ points of $A$,
which is excluded since $A$ has no three points on a line.

We can conclude that the $h$-vector of $Z$ is $(1,3,3,3,3,1)$, so $h_Z(3)=h_Z(4)=3$. Then, 
by \cite[Theorem 3.6]{BGM1994} there exists a cubic curve $C$ which contains a subset $Z'$ of $Z$ whose
ideal coincides with the ideal of $C$ up to degree $4$. If $C$ is a plane curve,
then its $h$-vector is $(1,2,3,3,3,\dots)$, so $Z'$ can miss at most $2$ points of $Z$,
which contradicts again the GLP of $A$. If $C$ spans $\Pj\C^4$, then the $h$-vector of $C$ 
is $(1,3,3,3,3,\dots)$ and the homogeneous ideal is generated in degree at most $3$. So,
if $C$ misses some points of $Z$, then $h_Z(3)>h_C(3)=3$, which is a contradiction.
Thus $C$ contains $Z$, hence it contains $A$.

It remains to show that $C$ is irreducible. $C$ cannot split in three lines, for one line
would then contain three points of $A$. If it splits in a line and an irreducible (plane) conic,
then either there exists a line containing three points of $A$, or $5$ points of $A$ lie in the
plane of the conic. Both situations contradict the GLP of $A$.
\end{proof}

\begin{proof}[Proof of \refprop{prop_disjointAB}]
Suppose that in \emph{every} alternative decomposition $B$ of cardinality equal to the rank $s\le 7$ of $\tensor{A}$ some of the points appear in both 
$A$ and $B$, say $A \cap B = \{[\vect{v}_1],\ldots,[\vect{v}_k]\}$ with $k > 0$. Then
\[
 \tensor{A} 
 = \vect{v}_1^{\circ 4} + \vect{v}_2^{\circ 4} + \cdots + \vect{v}_7^{\circ 4}
 = \lambda_1 \vect{v}_1^{\circ 4} + \cdots + \lambda_k \vect{v}_k^{\circ 4} + \vect{w}_{k+1}^{\circ 4} + \cdots + \vect{w}_s^{\circ 4}.
\]
It follows that
\[
 \tensor{A}' = (1-\lambda_1)\vect{v}_{1}^{\circ 4} + \cdots + (1-\lambda_k) \vect{v}_k^{\circ 4} + \vect{v}_{k+1}^{\circ 4} + \cdots + \vect{v}_7^{\circ 4} 
= \vect{w}_{k+1}^{\circ 4} + \cdots + \vect{w}_s^{\circ 4}.
\]

If any of the $\lambda_j$ are equal to $1$, then $\tensor{A}'$ would be an identifiable tensor because 
of Kruskal's theorem and the assumption that $A$ is in GLP. It follows that $s\ge7$, hence, $s=7$.
Comparing the lengths of the decompositions of $\tensor{A}'$, it follows that all $\lambda_j = 1$. But then 
$\{[\vect{v}_{k+1}], \ldots, [\vect{v}_7]\} = \{ [\vect{w}_{k+1}], \ldots, [\vect{w}_7]\}$ because of the identifiability
of $\tensor{A}'$. 
This implies the decompositions $A$ and $B$ of $\tensor{A}$ consist of the same set of points: $A = B$. 
By the assumption on minimality of $A$, it follows that $\tensor{A}$ is identifiable as well,
which contradicts our assumption. 

So, none of the $\lambda_j$ are equal to $1$. Then $\tensor{A}'$ has two decompositions, $A$ is still in 
GLP, and we let $B' = B \setminus A = \{[\vect{w}_{k+1}], \ldots, [\vect{w}_s]\}$ and $Z' = A \cup B'$. 
Applying \refprop{prop_twisted_cubic} to $Z'$ yields that $\tensor{A}'$  has alternative decompositions 
only if $|B'| = 7$, requiring $s \ge 8 \not\le 7$, contradicting the assumption that $B$ was of minimal cardinality. 

This proves that if $\tensor{A}$ is not $7$-identifiable with $A$ in GLP, then there must exist at least
one set of points $B$ such that $A \cap B = \emptyset$ and $\tensor{A} \in \langle v_d(A) \rangle \cap \langle v_d(B)\rangle$.
\end{proof}

\begin{proposition} \label{prop_criterion}
If $A$ is contained in an irreducible  rational twisted cubic curve $C$,
then $\tensor{A}$ is not identifiable, and the given  decomposition of $\tensor{A}$
is contained in a positive dimensional family of decompositions. In other words,
there exists a positive dimensional family of subsets $A_t$ of cardinality $7$
in $v_4(\Pj\C^4)$, with $A_0=A$, such that  $\tensor{A}$ belongs to the span 
of each $v_4(A_t)$.
\end{proposition}
\begin{proof}
The twisted cubic is itself the image of a Veronese map
$C=v_3(\Pj\C^2)$, thus $v_4(C)=v_{12}(\Pj\C^2)$ is a rational normal curve in $\Pj\C^{13}$.
The secant variety $\Sec{12}{\Pj\C^2}$ covers $\Pj\C^{13}$ and every rank-$7$ point of $\Pj\C^{13}$
is contained in an $1$-dimensional family of $7$ secant spaces. Thus when
$A$ is contained in a twisted cubic, then $\tensor{A}$ lies into the space $\Pj\C^{13}$ 
spanned by $v_4(C)=v_{12}(\Pj\C^2)$ and consequently it has infinitely many 
decompositions as a sum of $7$ tensors of rank $1$, lying in $v_4(C)$. Thus
there exists a $1$-dimensional family of decompositions for $\tensor{A}$ which includes $A$.
\end{proof}

Verifying that there does not exist a positive dimensional family of alternative decompositions over $\F$ may
be accomplished by exploiting the following result, which is essentially implicit in Terracini's paper \cite{Terracini1911}. 

\begin{lemma}\label{lem_terracini_criterion}
Let $\Var{V} \subset \F^N$ be an affine variety that is not $r$-defective. Let the points $p_1, \ldots, p_r \in \Var{V}$, and let 
$\Tang{p_i}{\Var{V}} \subset \F^N$ denote the affine tangent space to $\Var{V}$ at $p_i$. If the $p_i$'s are contained in a
family of decompositions of positive dimension, then $\dim \langle \Tang{p_1}{\Var{V}}, \ldots, \Tang{p_r}{\Var{V}} \rangle < \dim \Sec{r}{\Var{V}}$. 
\end{lemma}
\begin{proof}
Let $p = \sum_{i=1}^{r} p_i(t)$ with $p_i(0) = p_i$ and $t$ in a neighborhood of zero be a smooth curve passing through the $p_i$'s along which $p$ remains constant. As $\Var{V}$ is a variety, the Taylor series expansion of this analytic curve is well-defined and by \cite[Lemma 2.1]{Landsberg2006} may be written as
\(
 p_i(t) = p_i + t p_i^{(1)} + t^2 p_i^{(2)} + \cdots
\)
with $p_i^{(1)} \in \Tang{p_i}{\Var{V}}$ and $p_i^{(k)} \in \F^N$. After grouping terms by powers of $t$, we have
\[
 p = p + t \sum_{i=1}^r p_i^{(1)} + t^2 \sum_{i=1}^r p_i^{(2)} + \cdots.
\]
Since this holds for all $t$ in a neighborhood of $0$, it follows that $\sum_{i=1}^r p_i^{(k)} = 0$ for all $k$. In particular the case $k=1$ entails that $\dim \langle \Tang{p_1}{\Var{V}},\ldots,\Tang{p_r}{\Var{V}}\rangle$ is strictly less than the expected dimension of $\Sec{r}{\Var{V}}$. By assumption on $\Var{V}$, this concludes the proof.
\end{proof}

By Terracini's Lemma \cite{Terracini1911} we know that if the $(p_1,\ldots,p_r)$ are generic and $\Var{V}$ is 
nondefective, then $\dim \langle \Tang{p_1}{\Var{V}}, \ldots, \Tang{p_r}{\Var{V}} \rangle = \dim \Sec{r}{\Var{V}}$ 
so that the foregoing lemma can effectively exclude the possibility that such a positive dimensional family exists. 
The next sufficient condition for specific $7$-identifiability in $S^4 \C^4$ is then obtained.

\begin{proposition}\label{prop_sufficient}
Let $\F = \R$ or $\C$.
Let $p_i = \lambda_i \sten{a}{i}{\circ 4} \in v_4(\F^4) \subset S^4\F^4$, with $\lambda_i \in \F$ and $\sten{a}{i}{}\in\F^4$ for $i=1,\ldots,7$, be given in the form of a \emph{factor matrix}
\(A=\left[\begin{smallmatrix}\lambda_1^{1/4}\sten{a}{1}{ } & \cdots & \lambda_7^{1/4}\sten{a}{7}{ } \end{smallmatrix}\right].\)
If $A$ is in GLP, $A\odot A\odot A\odot A$ is of rank $7$, and there does not exist a family of alternative complex decompositions
passing through $A$, then $\tensor{A}=\sum_{i=1}^7 \lambda_i \sten{a}{i}{\circ 4} \in S^4\F^4$ is $7$-identifiable over $\C$, and, hence, $7$-identifiable over $\F$.
\end{proposition}
\begin{proof}
For $\F=\C$, we can assume without loss of generality that all $\lambda_i = 1$. The result then follows from \refprop{prop_criterion}.

For $\F=\R$, it suffices to note that we can apply \refprop{prop_criterion} to every complex decomposition of length $7$, 
in particular we can apply it to the right-hand side of
\[
 \tensor{A} = \sum_{i=1}^7 \lambda_i \sten{a}{i}{\circ 4} = \sum_{i=1}^7 (\lambda_i^{1/4} \sten{a}{i}{})^{\circ 4} \in S^4 \R^4,
\]
which in general is a complex Waring decomposition. If the conditions of the proposition are satisfied for the decomposition on the right-hand side, then it also proves that the corresponding real Waring decomposition is the unique complex decomposition of $\tensor{A}$, and, hence, it is the unique decomposition over both $\R$ and $\C$.
\end{proof}

\subsection{The algorithm}
Assume that we are given a decomposition
\[
\tensor{A} = \sum_{i=1}^r p_i = \sum_{i=1}^r \lambda_i \sten{a}{i}{\circ 4} \in S^4 \F^4
\]
with $\lambda_i \in \F$ and $\sten{a}{i}{ } \in \F^4$, $i=1,\ldots,r$, in the form of a matrix $A = \left[\begin{smallmatrix}\lambda_1^{1/4}\sten{a}{1}{ } & \cdots & \lambda_1^{1/4}\sten{a}{r}{ }\end{smallmatrix}\right] \in \C^{4 \times r}$. Then, the following steps should be taken.

\begin{enumerate}
 \item[S1.] If $r \ge 8$, the algorithm terminates claiming that it can not prove the identifiability of $\tensor{A}$.
 \item[S2.] If $r = 1$, the algorithm terminates and if $\sten{a}{1}{} \ne 0$ it states that $\tensor{A}$ is $1$-identifiable; otherwise if $\sten{a}{1}{} = 0$, it states that $\tensor{A}$ is not $1$-identifiable.
 \item[S3.] If $2 \le r \le 6$, perform the following steps:
 \begin{enumerate}
  \item[S3a.] Compute the Kruskal ranks $\kappa_1$ and $\kappa_2$ of $A$ and $A \odot A$ respectively.
  \item[S3b.] If $r \le \kappa_1 + \tfrac{1}{2} \kappa_2 - 1$, then the algorithm terminates stating that $\tensor{A}$ is $r$-identifiable. Otherwise it terminates, unable to prove identifiability.
 \end{enumerate}
 \item [S4.] If $r = 7$, perform the following steps:
 \begin{enumerate}
  \item[S4a.] Compute $A \odot A \odot A \odot A$ and verify that its rank equals $7$. If it does not, the algorithm terminates stating that $\tensor{A}$ is not $7$-identifiable.
  \item[S4b.] Compute the Kruskal rank of $A$. If it is not $4$, the algorithm terminates claiming that it cannot prove identifiability.
  \item[S4c.] Let $T_i$ be a basis for the tangent space $\Tang{p_i}{v_4(\C^4)}$. Compute the rank of $T = \left[\begin{smallmatrix} T_1 & \cdots & T_r \end{smallmatrix}\right]$. If it does not equal $21$, then the algorithm terminates claiming that it cannot prove $7$-identifiability.
  \item[S4d.] The algorithm terminates, stating that $\tensor{A}$ is $7$-identifiable.
 \end{enumerate}
\end{enumerate}

The ancillary file \texttt{identifiabilityS4C4.m2} contains an implementation of this algorithm in Macaulay2 \cite{M2}. 

\begin{proof}[Proof of \refthm{thm_effective_crit}]
The fact that the above algorithm is effective for all tensors in $S^4 \F^4$ is trivial for $r=1$; it follows from \refcor{cor_symm_reshaped_Kruskal} for $2 \le r \le 6$; for $r=7$ it follows from the fact that the assumptions leading to \refprop{prop_criterion}, namely GLP and minimality, fail only on Zariski-closed sets as well as the fact that $\Pj v_4(\C^4)$ is not defective for $r=7$ \cite{AH1995} so that the dimension condition in \reflem{lem_terracini_criterion} is only satisfied on a Zariski-closed set---effectiveness in the real case follows from the foregoing, \cite{QCL2016} and the fact that $\Pj v_4(\C^4)$ is not $7$-defective; and for $r\ge8$ the generic element of $\Sec{r}{\Var{V}}$ is not complex $r$-identifiable.
\end{proof}

\subsection{Two example applications of the algorithm}
We present two cases illustrating the foregoing algorithm in the original case $r = 7$.
\paragraph{An identifiable example}
Consider a real Waring decomposition of length $7$ that was randomly generated in Macaulay2:
\[
\tensor{A} = \sum_{i=1}^7 \sten{a}{i}{\circ 4},
\quad\text{with }
A = 
\begin{bmatrix}
 \sten{a}{i}{}
\end{bmatrix}_{i=1}^7
=
\begin{bmatrix}
 5  & -3 & 1 &  7 &  3 &  1 & -9 \\
 0  & 9  & 1 &  2 &  8 & -2 &  6 \\
 -8 & 5  & 5 & -3 & -4 & -6 & -8 \\
 3  & 7  & 9 & -3 &  8 &  7 & -7 
\end{bmatrix}.
\]
Executing the algorithm, we can skip steps S1--S3 and immediately move to step S4a. Using the functions in the \texttt{reshapedKruskal.m2} ancillary file,
the rank of $A\odot A\odot A\odot A$ is computed by \verb|rank(kr(A,{0,0,0,0}))|. It is $7$, so we proceed with step S4b. The 
Kruskal rank of $A$, which consists of computing the rank of $35$ $7\times7$ matrices, is $4$, as determined by the code fragment \texttt{kruskalRank(A)}. In step S4c, we compute rank of the $35 \times 28$ matrix $T$ whose columns span a subspace of the tangent space to $\Sec{r}{\Var{S}_\C}$ at $\tensor{A}$. The rank of this matrix is the maximal value $28$, so by \refprop{prop_sufficient} we may conclude that there is just one complex Waring decomposition. Since we started from a real decomposition, it follows that this is the unique Waring decomposition of $\tensor{A}$.

\paragraph{A nonidentifiable example}
The following classical lemma gives infinitely many Waring decompositions of the degree $12$ binary form $(x^2+y^2)^6$. The seven summands correspond to seven consecutive vertices of a regular $14$-gon in the Euclidean plane with coordinates $(x,y)$.
\begin{lemma}[{Reznick \cite[Theorem 9.5]{Rez1992}}]\label{lem:rez}
Let $R=2^{-12}\left[7\binom{12}{6}\right]$.
Then $\forall\phi\in{\mathbb R}$:
$$
\sum_{k=0}^6\left[\cos\Bigl(\frac{k\pi}{7}+\phi\Bigr)x+\sin\Bigl(\frac{k\pi}{7}+\phi\Bigr)y\right]^{12}=R(x^2+y^2)^6.
$$
These decompositions are minimal, in the sense that $\mathrm{rank}_{\C}\left[(x^2+y^2)^6\right] = 7.$
\end{lemma}

From the previous lemma we get the following example with infinitely many decompositions
of a rank $7$ symmetric tensor in $\R^4\otimes\R^4\otimes\R^4\otimes\R^4$.
Let $z_0,\ldots, z_3$ be coordinates in $\R^4$ and let $A_{k,\phi}=\cos^3(\frac{k\pi}{7}+\phi)z_0+\cos^2(\frac{k\pi}{7}+\phi)\sin(\frac{k\pi}{7}+\phi)z_1+
\cos(\frac{k\pi}{7}+\phi)\sin^2(\frac{k\pi}{7}+\phi)z_2+\sin^3(\frac{k\pi}{7}+\phi)z_3$ be a linear form in $\Pj\R^4$.
These linear forms correspond to points on the twisted cubic curve parametrized by
$z_i= x^{3-i}y^i$ in the dual space.
Now define 
\begin{equation}\label{eq:infinity7}
\tensor{A}=\sum_{k=0}^6 \vect{a}_{k,\phi}^{\circ 4} \quad\text{with } \sten{a}{k,\phi}{} = 
\begin{bmatrix}
\cos^3(\frac{k\pi}{7}+\phi) \\
\cos^2(\frac{k\pi}{7}+\phi)\sin(\frac{k\pi}{7}+\phi) \\
\cos(\frac{k\pi}{7}+\phi)\sin^2(\frac{k\pi}{7}+\phi) \\
\sin^3(\frac{k\pi}{7}+\phi)
\end{bmatrix}.
\end{equation} 
Then $\tensor{A}$ is a symmetric tensor in $\R^4\otimes\R^4\otimes\R^4\otimes\R^4$ (or equivalently a quartic polynomial) which does not depend on $\phi$ by \reflem{lem:rez}.
For every $\phi$, \refeqn{eq:infinity7} is a different Waring decomposition with seven summands of $\tensor{A}$.

We now apply the algorithm to this example, where we have chosen $\phi=0$ as particular decomposition to be handed to the algorithm. It will be necessary to perform numerical computations as $\tensor{A}$ no longer admits coordinates over the integers. The $\epsilon$-rank of a matrix is defined as the number of singular values that are larger than $\epsilon$; the rank of a matrix is its $0$-rank. There always exists a positive $\delta > 0$ such that all the $\delta'$-ranks of a matrix are equal for all $0 \le \delta' \le \delta$. Through a perturbation analysis this value of $\delta$ can usually be determined. However, in this brief example such a rigorous approach is not pursued. We will choose $\delta$ very small and hope that the $\delta$-rank and the rank coincide. In Macaulay2, the $\epsilon$-rank can be computed with the \texttt{numericalRank} function from the \texttt{NumericalAlgebraicGeometry} package. In our experiment, we used the completely arbitrary choice $\epsilon=10^{-12}$. Running the algorithm, it immediately skips steps S1, S2 and S3. The numerical rank of $A\odot A\odot A\odot A$ was determined to be $7$ in step S4a (the largest and smallest singular values were approximately $1.08615$ and $0.21978$ respectively). In step S4b, the (numerical) Kruskal rank was $4$. Computing the singular values of $T$ in step S4c resulted in the following values:
\begin{align*}\small
\begin{array}{lllllll}
27.4692;   &27.3073;    &8.70636;   &8.59365;    &7.26970;    &7.11095;   &7.02903;   \\
\phantom{0}6.83427;   &\phantom{0}4.05864;    &3.89601;   &3.01363;    &2.45649;    &2.24154;   &2.07335;   \\ 
\phantom{0}1.90712;   &\phantom{0}1.90496;    &1.58224;   &1.52450;    &1.35632;    &1.26918;   &1.00762;   \\
\phantom{0}0.553879;  &\phantom{0}0.481666;   &0.424916;  &0.364948;   &0.175228;   &0.165698;  &6.60364 \cdot 10^{-16}.
\end{array}
\end{align*}
The numerical rank is only $27 < \dim \Sec{7}{v_4(\C^4)} = 28$. So the algorithm terminates claiming that it cannot prove $7$-identifiability of $\tensor{A}=\sum_{k=0}^6 \sten{a}{k,0}{\circ 4}$. As $\tensor{A}$ has a family of decompositions of positive dimension, this was expected.

\section{Application to algorithm design}\label{sec_applications}
An important consequence of \refthm{thm_reshaped_kruskal} is that it provides a solid theoretical foundation for algorithms computing tensor rank decompositions based on reshaping, such as \cite{PTC2013,BCMV2014}. These algorithms attempt to recover a tensor rank decomposition of a rank-$r$ tensor as in \refeqn{eqn_rank_decomposition},
living in $\F^{n_1} \otimes \F^{n_2} \otimes \cdots \otimes \F^{n_d}$, by considering $\tensor{A}$ as an element of $\F^{\Pi_\vect{h}} \otimes \F^{\Pi_\vect{k}} \otimes \F^{\Pi_\vect{l}}$ with $\vect{h}\party\vect{k}\party\vect{l} =\{1,2,\ldots,d\}$ and instead computing a decomposition 
\begin{align}\label{eqn_reshaped_decomposition}
 \tensor{A} = \sum_{i=1}^r \sten{b}{i}{1} \otimes \sten{b}{i}{2} \otimes \sten{b}{i}{3}.
\end{align}
If both decompositions \refeqn{eqn_rank_decomposition} and \refeqn{eqn_reshaped_decomposition} are unique, then the rank-$1$ tensors satisfy
\[
 \sten{b}{\sigma_i}{1} = \sten{a}{i}{h_1} \otimes \sten{a}{i}{h_2} \otimes \cdots \otimes \sten{a}{i}{h_{s}},\;
 \sten{b}{\sigma_i}{2} = \sten{a}{i}{k_1} \otimes \sten{a}{i}{k_2} \otimes \cdots \otimes \sten{a}{i}{k_{t}},\; \text{ and }
 \sten{b}{\sigma_i}{3} = \sten{a}{i}{l_1} \otimes \sten{a}{i}{l_2} \otimes \cdots \otimes \sten{a}{i}{l_{u}},
\]
for some permutation $\sigma$ of $\{1,2,\ldots,r\}$, and where $s$, $t$, and $u$ are the cardinalities of $\vect{h}$, $\vect{k}$, and $\vect{l}$ respectively. 
One of the advantages of this approach is that decomposition \refeqn{eqn_reshaped_decomposition} could be computed using one of the direct methods that only\footnote{There also exists an algorithm due to Bernardi, Brachat, Comon, and Mourrain \cite{BBCM2013} for computing a tensor rank decomposition of any tensor, but in general it requires the solution of a system of linear, quadratic and cubic equations.} exist for third-order tensors, e.g., \cite{DdL2014,dL2006,DdL2016}. Thereafter, decomposition \refeqn{eqn_rank_decomposition} can be efficiently recovered by computing rank-$1$ decompositions of the vectors $\sten{b}{i}{k}$ for all $k=1,2,3$ and $i=1,2,\ldots,r$, using one of several suitable algorithms, such as \cite{VVM2012,SRK2009,Zhang2001,EGH2009}.

The conditions under which aforementioned algorithms are expected to recover the decomposition \refeqn{eqn_rank_decomposition} have not been studied. 
This is precisely the problem that \reflem{lem_max_glp} and \refthm{thm_reshaped_kruskal} tackle: if \refeqn{eqn_rank_decomposition} is a generic decomposition with $r$ satisfying bound \refeqn{eqn_r_bound}, then decompositions \refeqn{eqn_rank_decomposition} and \refeqn{eqn_reshaped_decomposition} are simultaneously unique in the spaces $\F^{n_1} \otimes \cdots \otimes \F^{n_d}$ and $\F^{\Pi_\vect{h}} \otimes \F^{\Pi_\vect{k}} \otimes \F^{\Pi_\vect{l}}$ respectively, entailing that the aforementioned reshaping-based algorithms can recover the unique decomposition \refeqn{eqn_rank_decomposition} of $\tensor{A}$ via \refeqn{eqn_reshaped_decomposition}. 

\section*{Application to Comon's conjecture}
An important corollary of the results in \refsec{sec_reshaped_kruskal} concerns a conjecture that is attributed to 
P.~Comon and appears explicitly in \cite[Sections 4.1 and 5]{CGLM2008}. Little progress has been 
made on this conjecture with some sparse results appearing in the literature \cite{BB2013,BL2013,FS2013,F2015}. 
We should remark at this point that the claim on page 321 of \cite{F2015} about \cite{COV2016} does not follow from the latter: \cite[Theorem 1.1]{COV2016} only states that the generic symmetric tensor of subtypical symmetric rank admits only one Waring decomposition, however this does not rule out the existence of shorter tensor rank decompositions. The results of \cite{COV2016} do not make claims about the correctness of Comon's conjecture.

The original formulation of Comon's conjecture is that a symmetric tensor that has a rank-$r$ Waring decomposition 
does not admit a shorter tensor rank decomposition. 
We confirm this conjecture for \emph{generic} symmetric tensors whose symmetric rank $r$ is small. 

\begin{theorem}[Comon's conjecture is generically true for small rank]\label{thm_comon}
Let $\F = \C$ or $\R$.
Let 
\[
p = \sum_{i=1}^r \psi_i \sten{a}{i}{} \otimes \cdots \otimes \sten{a}{i}{}
\]
with $\psi_i \in \F_0$ and $\sten{a}{i}{}\in\F^{n+1}$ be a generic $d$th order 
symmetric tensor in $\F^{n+1} \otimes \cdots \otimes \F^{n+1}$ of symmetric rank $r$. If
\[
 r \le 
 \begin{cases}
 \tfrac{3}{2}n - 1 &\text{if } d = 3 \\
 \binom{k + n}{n} + \frac{1}{2} \binom{2+n}{n} - 1 &\text{if } d = 2k + 1 \\
 \binom{k + n}{n} - n - 1 &\text{if } d = 2k,
 \end{cases}
\]
with $2 \le k \in \mathbb{N}$, then $p$ admits only Waring decompositions.
In particular, the symmetric rank and the tensor rank of $p$ coincide.
\end{theorem}
\begin{proof}
The odd cases follow from \refcor{cor_symm_reshaped_Kruskal}.

The even case follows from considering the square flattening of $p$:
\[
 p_{(1,\ldots,k)} = \sum_{i=1}^r \psi_i \sten{a}{i}{\otimes k} (\sten{a}{i}{\otimes k})^T = A \Psi A^T
\]
where $A = \begin{bmatrix} \sten{a}{i}{\otimes k} \end{bmatrix}_{i=1}^r \in \F^{(n+1)^k \times r}$ and $\Psi = \operatorname{diag}(\psi_1, \ldots, \psi_r)$.
By \reflem{lab_symm_glp}, the points $\sten{a}{i}{\otimes k}$ are in GLP in $S^k \F^{n+1}$ so that 
$\rank{A} = \min \{r, \binom{k+n}{k}\} = r$. It follows from Sylvester's rank inequality that the matrix rank of 
$p_{(1,\ldots,k)}$ is $r$. Assume that $p$ has an alternative tensor rank decomposition of rank $s \le r$, i.e.,
\[
 p = \sum_{i=1}^s \varphi_i \sten{a}{i}{1} \otimes \cdots \otimes \sten{a}{i}{d},\; \text{and let }
 p_{(1,\ldots,k)} = \sum_{i=1}^s (\varphi_i \sten{a}{i}{1} \otimes \cdots \otimes \sten{a}{i}{k})(\sten{a}{i}{k+1} 
\otimes \cdots \otimes \sten{a}{i}{d})^T = BC^T,
\]
where $B = \begin{bmatrix} \varphi_i \sten{a}{i}{1} \otimes \cdots \otimes \sten{a}{i}{k} \end{bmatrix}_{i=1}^r$ and 
$C = \begin{bmatrix} \sten{a}{i}{k+1} \otimes \cdots \otimes \sten{a}{i}{d} \end{bmatrix}_{i=1}^r$.
Since $p_{(1,\ldots,k)}$ has matrix rank $r$, it follows that $r = s$. In particular, $C$ has a set of linearly 
independent columns, and hence it has a well-defined left inverse of the form $C^\dagger = (C^H C)^{-1} C^H$. 
It follows from
\(
 p_{(1,\ldots,k)} = A \Psi A^T = B C^T
\) 
that \(A \Psi A^T (C^\dagger)^T = B\). Since the image of $A$ is contained in $S^k \F^{n+1}$, it follows 
that the columns of $B$ are in fact symmetric rank-$1$ tensors in $S^k \F^{n+1}$, each of which being a 
linear combination of the points $\sten{a}{i}{\otimes k}$. However, as the $\sten{a}{i}{}$ are generic, 
it follows from the trisecant lemma \cite[Proposition 2.6]{ChCi2001} that $\langle \sten{a}{1}{\otimes k}, 
\sten{a}{2}{\otimes k}, \ldots, \sten{a}{r}{\otimes k} \rangle$ does not intersect $\Sec{r}{\Pj v_{k}(\F^{n+1})}$ 
at any other points than the $[\sten{a}{i}{\otimes k}]$'s if $r$ satisfies the bound in the formulation of the corollary.
Therefore, there exists a permutation matrix\footnote{A matrix whose columns are a permutation of the identity matrix.} $P$ 
and a nonsingular diagonal matrix $\Lambda$ such that $B = A P \Lambda$. 
Note that $A$ has a set of linearly independent columns so that $A^\dagger = (A^H A)^{-1} A^H$ is also 
well defined.
Then, applying this left inverse to
\(
 A \Psi A^T = B C^T = A P \Lambda C^T
 \) yields 
\(
\Psi A^T = P \Lambda C^T,
\)
so that $C = A \Psi P \Lambda^{-1}$. Explicitly,
\[
 B = \begin{bmatrix} 
      \lambda_1 \sten{a}{\pi_1}{\otimes k} & \cdots & \lambda_r \sten{a}{\pi_r}{\otimes k}
     \end{bmatrix}
\quad\text{and}\quad
 C = \begin{bmatrix} 
      \lambda_1^{-1} \psi_{\pi_1} \sten{a}{\pi_1}{\otimes k} & \cdots & \lambda_r^{-1} \psi_{\pi_r} \sten{a}{\pi_r}{\otimes k}
     \end{bmatrix},
\]
where $\pi$ is the permutation represented by $P$,
so that the supposed alternative tensor rank decomposition of $p$ is also a Waring decomposition.
\end{proof}

This result asymptotically improves the best known range, which to our knowledge is \cite[Theorem 7.6]{FS2013} (only stated for $\F=\C$), 
by a factor of $\mathcal{O}(n)$ when $d$ is even and $n$ is large. For odd $d$, the improvement of \refthm{thm_comon} occurs only in the lower-order terms.

The following identifiability result is immediate from the above proof. 
\begin{corollary}
The generic tensor of rank $r$ whose rank satisfies the bounds in \refthm{thm_comon} has a unique Waring decomposition over $\F$ that is also its unique tensor rank decomposition. 
\end{corollary}

The last statement proves more than Comon's conjecture as it states that the
generic $p \in S^d \F^{n+1}$ of symmetric rank $r$ is simultaneously $r$-identifiable with respect to the Veronese 
variety $v_d(\Pj\F^{n+1})$ and the Segre variety $\operatorname{Seg}(\Pj\F^{n+1} \times \cdots \times \Pj\F^{n+1})$. 
We suspect that this may be true for larger values of $r$ as well.

\section{Conclusions}\label{sec_conclusions}
An important quality measure for criteria for identifiability of tensors is its effectiveness. 
\refthm{thm_reshaped_kruskal} proves that the popular Kruskal criterion when it is combined with reshaping is effective. The proof yielded insight into reshaping-based algorithms for computing tensor rank decompositions, proving that they will recover the unique 
decomposition with probability $1$ if the rank is within the range of effectiveness of \refthm{thm_reshaped_kruskal}.
The range of effectiveness for symmetric identifiability of the reshaped Kruskal criterion was established.
Combining this result with results from the literature established that a small number of low-dimensional symmetric tensor spaces are effectively identifiable (for all ranks).
By analyzing the Hilbert function, we could prove effective identifiability of an additional case, namely $S^4 \F^4$. To our knowledge, \refthm{thm_effective_crit} lists all proven instances of effectively identifiable spaces.

All criteria for specific $r$-identifiability that we are aware of for 
order-$d$ tensors are applicable for $r$ up to about $\mathcal{O}(n^{d/2})$ when $n_1 = \cdots = n_d = n$, whereas 
generic $r$-identifiability is expected to hold up to $\mathcal{O}(n^{d-1})$. We believe that this 
gap is related to the fact that nonidentifiable points on a generically $r$-identifiable variety where 
Terracini's matrix is of maximal rank and the Hessian criterion \cite[Theorem 4.5]{COV2014} is satisfied must be singular points of the variety by 
\cite[Lemma 4.4]{COV2014}. Characterizing the singular locus of secant varieties
is a difficult problem. The approach we suggested based on the Hilbert function 
has the advantage that it sidesteps the problem of smoothness by proving that there are no \emph{isolated} 
unidentifiable points when the assumptions of \refprop{prop_sufficient} are satisfied.
It is an open question insofar the analysis of the Hilbert function may be more generally applicable for 
proving that the unidentifiable points must be contained in a curve; some results in this direction were established in \cite{BaC2013,BC2012}. Isolated
unidentifiable tensors also exist in some cases, as was shown in \cite[Example 3.4]{BaC2013}.

\section*{Acknowledgements}
We thank I. Domanov and J. Migliore for helpful discussions.

\providecommand{\bysame}{\leavevmode\hbox to3em{\hrulefill}\thinspace}
\providecommand{\MR}{\relax\ifhmode\unskip\space\fi MR }
\providecommand{\MRhref}[2]{%
  \href{http://www.ams.org/mathscinet-getitem?mr=#1}{#2}
}
\providecommand{\href}[2]{#2}

\end{document}